\documentclass[12pt,reqno]{amsart}
\usepackage{amssymb}
\usepackage{amsfonts}
\usepackage{amsmath}
\usepackage{mathtools} 
\usepackage{color}
\usepackage{graphicx, float}
\usepackage{tikz}
\usepackage{cite}
\usepackage{epsfig,mathrsfs}
\usepackage[bf,SL,BF]{subfigure}
\usepackage{fancyhdr}
\usepackage{marginnote}

\newtheorem{Coro}{Corollary}
\newtheorem{lemma}{Lemma}

\newtheorem{remark}{Remark}

\newcommand{\Z}{\mathbb{Z}}
\newcommand{\R}{\mathbb{R}}
\newcommand{\C}{\mathbb{C}}
\newcommand{\re}{\mbox{Re}}

\newcommand{\bfx}{\boldsymbol{x}}
\newcommand{\veck}{\boldsymbol{k}}
\newcommand{\vecx}{\boldsymbol{x}}
\newcommand{\vecalpha}{\boldsymbol{\alpha}}
\newcommand{\vecgamma}{\boldsymbol{\gamma}}
\newcommand{\vecbeta}{\boldsymbol{\beta}}
\newcommand{\svecbeta}{\boldsymbol{\scriptscriptstyle
 \beta}}
\newcommand{\vecE}{\boldsymbol{E}}
\newcommand{\vecC}{\boldsymbol{C}}
\newcommand{\vecI}{\boldsymbol{I}}
\newcommand{\vecf}{\boldsymbol{f}}
\newcommand{\vecg}{\boldsymbol{g}}
\newcommand{\vecF}{\boldsymbol{F}}
\newcommand{\vecL}{\boldsymbol{L}}
\newcommand{\vecH}{\boldsymbol{H}}
\newcommand{\vecX}{\boldsymbol{X}}
\newcommand{\vecK}{\boldsymbol{K}}

\newcommand{\vecu}{\boldsymbol{u}}

\newcommand{\vecw}{  \boldsymbol{w} }
\newcommand{\vecv}{  \boldsymbol{v}  }
\newcommand{\vecU}{\boldsymbol{U}}
\newcommand{\vecV}{\boldsymbol{V}}

\newcommand{\mdiv}{\mbox{div\,}}
\newcommand{\mcurl}{\mbox{curl\,}}

\graphicspath{{pics/}}

\begin{document}
\title{Determination of electromagnetic Bloch modes in a medium with frequency-dependent coefficients}
\author{C. Lackner\address{Institute for Analysis and Scientific Computing Wiedner Hauptstrasse 8-10 1040 Wien, Austria ({\tt{}christopher.lackner@tuwien.ac.at})}, S. Meng\address{Department of Mathematics, 2074 East Hall
530 Church Street
Ann Arbor, MI 48109-1043, USA ({\tt{}shixumen@umich.edu})} and P. Monk\address{Department of Mathematical Sciences, University of Delaware, Newark DE 19716, USA. ({\tt{}monk@udel.edu.})}}
\date{}
\maketitle

\begin{abstract}
We provide a functional framework and a numerical algorithm to compute the Bloch variety for Maxwell's equations when the electric permittivity is frequency dependent. We incorporate the idea of a mixed formulation for Maxwell's equations to obtain a quadratic eigenvalue for the wave-vector in terms of the frequency. We reformulate this
problem as a larger linear eigenvalue problem and prove that this results in the need to compute eigenvalues of a compact operator.  Using finite elements, we provide preliminary numerical examples of the scheme for both 
frequency independent and frequency dependent permittivity.
\end{abstract} 

{\bf Keywords}: Bloch variety, quadratic eigenvalue, composite materials, frequency-dependent materials.
\section{Introduction}  \label{Introduction}
Photonic crystals are engineered periodic structures designed to manage light (see for example \cite{toader2004photonic,JJWM}).  In particular, it is important to design
materials having band gaps: these are intervals of frequencies for which there is an absence of wave
propagation in any direction.  One way to quantify the band gap is via the dispersion relation or, more generally, the Bloch variety which represents the relationship between a possibly complex-valued wave vector and a possibly complex-valued frequency as outlined below.
Band gap information and the Bloch variety have applications in device design. We refer to \cite{kuchment2016overview} as well as the textbook \cite{JJWM} for more details.  

To fix ideas, let us now describe the electromagnetic Bloch variety problem in more detail.
We consider the propagation of electromagnetic waves in periodic media in $\R^3$. The electric field $\vecE$ and magnetic field $\vecH$ satisfy Maxwell's equations
\begin{eqnarray} \label{MengMonk2016MaxwellEqn}
\mcurl \vecE - i \omega \mu \vecH = 0, \quad \mcurl \vecH + i \omega \epsilon \vecE = 0,
\end{eqnarray}
where  $\epsilon \in L^\infty(\R^3)$ is the electric permittivity, $\mu \in L^\infty(\R^3)$ is the magnetic permeability, and $\omega$ is the angular frequency. We consider the case that the electric permittivity is allowed to be {\it frequency dependent} (and depend on position) so $\epsilon=\epsilon(x,\omega)$ where $x$ denotes position in $\mathbb{R}^3$,
and assume that $\mu=1$ since the relevant materials are not generally magnetic. In addition we assume that $\Re{\epsilon}$ is uniformly bounded below away from zero. 

The medium is assumed to have unit periodicity on a cubic lattice. The first Brillouin zone is assumed to be $[-\pi, \pi]^3$. Let $\Z = \{ 0, \pm 1, \pm 2, \dots\}$ and $\Lambda = \Z ^3$, we have
\begin{eqnarray*}
\epsilon(\vecx + n,\omega) = \epsilon (\vecx,\omega), \quad \mbox{{for a.e. }} \vecx \in \R^3, n \in \Lambda.
\end{eqnarray*}
We define the periodic domain as the quotient space $\Omega= \R^3 / \Lambda$. We remark that $\Omega$ has no boundary. 

Let $\vecu$ be defined by $\vecH(\bfx) =e^{i \veck \cdot \vecx} \vecu(\bfx)$ where $\veck$ is a given wave vector, then $\vecu(\bfx)$ is periodic and equation (\ref{MengMonk2016MaxwellEqn}) can be reduced to 
\begin{eqnarray} 
\mcurl_{\veck} \left( \frac{1}{\epsilon} \mcurl_{\veck} \vecu \right) = \omega^2 \vecu &\quad \mbox{in} \quad& \Omega, \label{MengMonk2016MaxwellCell1}\\
\mdiv_{\veck} \vecu = 0 &\quad \mbox{in} \quad& \Omega,\label{MengMonk2016MaxwellCell2}
\end{eqnarray}
where we use the following short-hand notation
\begin{eqnarray*}
\begin{aligned}
\mcurl_{\veck}  &= \mcurl + i {\veck} \times, \\
\mdiv_{\veck} &= \mdiv + i {\veck} \cdot,\\
\nabla_{\veck} &= \nabla + i {\veck}.
\end{aligned}
\end{eqnarray*}
The Bloch variety is the set of all pairs $\left(\veck, \omega \right)$  such that there exists a non-trivial periodic solution $\vecu$ to equations (\ref{MengMonk2016MaxwellCell1})-(\ref{MengMonk2016MaxwellCell2}). For more details we refer to \cite{kuchment2016overview}.

Usually the Bloch variety is computed assuming that the material in the photonic crystal has a real permittivity that is independent of frequency.  This is done by choosing the wave-vector $\veck$ above. Then equations (\ref{MengMonk2016MaxwellCell1})-(\ref{MengMonk2016MaxwellCell2}) becomes a linear eigenvalue problem for the
eigenpair $(\omega^2,\vecu)$. {Computing}
all possible values of $\omega$ {then reduces to finding} the eigenvalues of a self-adjoint 
compact operator.  This eigenvalue problem can be solved by discretizing the equations in the usual way using {conforming} edge finite elements to discretize $\vecu$ and vertex elements to discretize the Lagrange multiplier that imposes the
divergence condition (see for example \cite{dobson2001analysis,dobson2000efficient,boffi2006interpolation}). Other discretizations  are possible: for example, a Fourier basis is used in the widely used
open source package MPB~\cite{MPB}.  

However, there is also significant interest in computing the Bloch variety of frequency-dependent materials, for novel applications in optical metamaterials and dispersive photonic crystals \cite{valentine2008three,soukoulis2007negative,toader2004photonic,brule2016calculation,alagappan2013optical,bai2013efficient,effenberger2012linearization,engstrom2014spectral,engstrom2009spectrum,engstrom2010complex,gu2006applications,hermann2008photonic,ivchenko2006resonant,kaso2007nonlinear,diaz2015thz,raman2010photonic,rybin2016inverse,serebryannikov2015effect,sozuer1994photonic,park1999three,toader2004photonic,zheng2017frequency}. Electronic and vibrational excitations in a material may interact resonantly with an electromagnetic wave and dramatically alter its propagation through the medium. 

For frequency dependent coefficients, an alternative to computing the frequency for a given wave vector is possible: the Bloch variety can be computed {by finding all wave vectors} $\veck$ for a given frequency $\omega$ (and hence a given value of
$\epsilon(x,\omega)$ throughout the domain). This results  in a quadratic eigenvalue problem (see \cite{effenberger2012linearization,engstrom2014spectral,engstrom2009spectrum,engstrom2010complex} for the case of acoustic, TE or TM waves) which will be the focus of this paper. We refer to \cite{tisseur2001quadratic,guttel2017nonlinear} and the reference therein for discussions and surveys {devoted to}
 nonlinear eigenvalue problem. Algorithms for finding the Bloch variety for frequency-dependent electromagnetic propagation in three-dimensional composite materials are much less developed {than for the frequency indpedent case}. Difficulties arise from, for instance, from the fact that the divergence free condition of the Maxwell system has to be respected. The mixed formulation in \cite{dobson2001analysis,dobson2000efficient,boffi2006interpolation} provides a functional framework within which the divergence free condition is handled properly, and we shall show that this framework can also be applied when computing the wave vector for a given frequency. We then linearize the quadratic eigenvalue problem using a mixed-quadratic formulation. This results in a larger, non-self adjoint eigenvalue problem which we solve by the Arnoldi method~\cite{arnoldi}.  It is the larger size of the numerical problem that is the main drawback of the method.

There are alternatives to using the quadratic eigenvalue approach of this paper. In \cite{roman18}, a Drude model is assumed for the frequency dependence of 
the permitivity of the medium in part of the unit cell.  This allows the Bloch mode problem to be converted into a non-linear
eigenvalue (obviously this approach can be extended to other rational approximations of the permittivity).  The  SLEPc  package (see \cite{SLEPc}) is then used to compute the eigenvalues.  Our approach avoids the need to model the permittivity by a function. Another alternative, the ``cutting surface'' method of \cite{toader2004photonic} uses multiple solutions of the standard approach {(fixing $\veck$ and computing $\omega$ with a frozen coefficient)} together  with an approximation scheme that uses
a  plane wave basis to compute the Bloch variety in the frequency dependent case.

The main contributions of this paper are: 1) to formulate a new stabilized quadratic eigenvalue problem for the electromagnetic Bloch variety calculation, 2) to prove that the resulting problem can be linearized resulting in 
a linear eigenvalue problem for a compact self adjoint operator, and 3) to provide some preliminary numerical examples that illustrate the behavior of our method.  Future work will include a more detailed numerical study.

The outline of this paper is as follows.  In Section {\ref{decom}}, we define the function spaces used in this paper, summarize the Fourier analysis of the problem, and recall an important regularity result.  In Section~\ref{mixed} we propose a variational formulation for the quadratic eigenvalue problem strongly related to that of
\cite{dobson2001analysis, boffi2006interpolation} but with an additional constraint that we have found to be necessary
for numerical stability in our case. We also give our linearized eigenvalue problem and show that this is equivalent to the
original quadratic problem.  In Section~\ref{linear} we show that the linearized problem results in an eigenvalue problem for a compact operator (and hence has a discrete spectrum).  In Section~\ref{num}, we then give two examples of numerical results using the linearized problem.  In particular we show that the new method agrees with a standard finite element calculation of the Bloch variety when applied to a frequency independent problem.  We also show results for a frequency dependent problem similar to one in \cite{toader2004photonic}.   For this problem we also investigate
the convergence rate numerically.  Finally in Section~\ref{concl} we present some conclusions.

In this paper vectors, vector functions and vector function space are shown in bold-face.
\section{Decomposition and regularity}\label{decom}
To begin with, we introduce the following periodic versions of the vector Sobolev spaces:
\begin{eqnarray*}
H^1_{p} (\Omega) &=& \{ f \in L^2(\Omega) \, : \, \nabla f \in \vecL^2(\Omega),\\&&\qquad\mbox{{with $f$ one-periodic in $x_1$,\,$x_2$ and $x_3$}}\}, \\
\vecH_p (\mcurl ; \Omega) &=& \{ \vecu\in \vecL^2(\Omega) \, : \, \mcurl \vecu \in \vecL^2(\Omega)\\&&\qquad\mbox{{with $\vecu$ one-periodic in $x_1$,\,$x_2$ and $x_3$}} \}, \\
\vecH_p(\mdiv ; \Omega) &=& \{   \vecu\in \vecL^2(\Omega) \, : \, \mdiv \vecu \in L^2(\Omega) \\&&\qquad\mbox{{with $\vecu$ one-periodic in $x_1$,\,$x_2$ and $x_3$}}\}.
\end{eqnarray*}
{In the above definitions, the statement that a given function is one periodic is to be interpreted as meaning that the 
one-periodic extension of the given function or vector is locally in the given Sobolev space on $\R^3$.} In the above definitions, the subscript $p$ represents the periodic version. 

Now we summarize a Fourier analysis of vector-valued functions, and refer to \cite{dobson2001analysis} for more details. Any sufficiently regular 1-periodic vector function $\vecw \in \vecL^2(\Omega)$ can be represented as
\begin{eqnarray}
\vecw  = \sum_{\vecI \in J} e^{i \vecI\cdot \vecx} \vecC_{\vecI}, \label{FB}
\end{eqnarray}
where $J=\{ 2\pi(i_1,i_2,i_3): \mbox{for integers} \,\,  i_1, i_2, i_3 \}$, and $\vecC_{\vecI}\in\mathbb{C}^3$ is a vector-valued constant.  The Sobolev spaces of periodic functions can be characterized as following,
\begin{eqnarray*}
\vecH^s_p = \{ \vecu \in \vecL^2(\Omega): \, \vecu  = \sum_{\vecI \in J} e^{i \vecI\cdot \vecx} \vecC_{\vecI} \quad \mbox{and} \quad \sum_{\vecI \in J} (1+|\vecI|^2)^s |\vecC_{\vecI}|^2 < \infty \},
\end{eqnarray*}
and an equivalent $\vecH^s_p$-norm is also given by
\begin{eqnarray*}
\left( \sum_{\vecI \in J} |\vecgamma^{\vecI}|^{2s} |\vecC_{\vecI}|^2 \right) ^{\frac{1}{2}},
\end{eqnarray*}
where  $\vecgamma^{\vecI}={\vecbeta} + \vecI$ with $\vecbeta\not=0$ and $\vecbeta \in [-\pi, \pi]^3$. For a vector valued function $\vecw  = \sum_{\vecI \in J} e^{i \vecI\cdot \vecx} \vecC_{\vecI}$, the following identities hold
\begin{eqnarray*}
\mcurl_{\vecbeta} \vecw &=& \sum_{\vecI \in J} e^{i \vecI\cdot \vecx} N_{\vecI} \vecC_{\vecI}, \\
\mcurl_{\vecbeta} \mcurl_{\vecbeta} \vecw  &=& \sum_{\vecI \in J} e^{i \vecI\cdot \vecx} N_{\vecI} (N_{\vecI} \vecC_{\vecI}), 
\end{eqnarray*}
where 
\begin{eqnarray*}
N_{\vecI}= i 
\left( \begin{array}{ccc}
0 & - \gamma_3^{\vecI} & \gamma_2^{\vecI} \\
\gamma_3^{\vecI}  & 0 & -\gamma_1^{\vecI} \\
- \gamma_2^{\vecI} & \gamma_1^{\vecI} & 0
\end{array} \right)
\end{eqnarray*}
with $\vecgamma^{\vecI}=(\gamma^{\vecI}_1,\gamma^{\vecI}_2,\gamma^{\vecI}_3)$. For any $\vecC_{\vecI} \in \C^3$, the following identities hold
\begin{eqnarray*}
|N_{\vecI} \vecC_{\vecI}|^2+|\vecgamma^{\vecI} \cdot \vecC_{\vecI}|^2=|\vecgamma^{\vecI}|^2|\vecC_{\vecI}|^2,
\end{eqnarray*}
and in particular since $N_{\vecI} \vecgamma^{\vecI}=0$,
\begin{eqnarray} \label{MLMMaxwellDecompositionEqn}
|N_{\vecI} (N_{\vecI}\vecC_{\vecI})|^2= |N_{\vecI} (N_{\vecI}\vecC_{\vecI})|^2+|\vecgamma^{\vecI} \cdot (N_{\vecI} \vecC_{\vecI})|^2=|\vecgamma^{\vecI}|^2|N_{\vecI} \vecC_{\vecI}|^2.
\end{eqnarray}
We also need the following lemma from \cite{dobson2001analysis}. Let $\|\cdot\|_s$ denote the $\vecH^s(\Omega)$-norm where $s$ is any non-negative number, and $\|\cdot\|$ conveniently denotes the $\vecL^2(\Omega)$-norm.
\begin{lemma} \label{MengMonk2016DecompositionVectorPeriodic}
Let $\vecbeta$ be a non-zero vector in the first Brillouin zone $[-\pi, \pi]^3$. Give $\vecu \in \vecL^2(\Omega)$ there exists unique functions $\vecw \in \vecH^1_{p}(\Omega)$ and $\phi \in H^1_p(\Omega)$ satisfying
\begin{eqnarray*}
\vecu = \mcurl_{\vecbeta} \vecw + \nabla_{\vecbeta} \phi \quad \mbox{and} \quad \nabla_{\vecbeta} \cdot \vecw=0.
\end{eqnarray*}
Furthermore,
\begin{eqnarray*}
\|\vecw\|_1 + \|\phi\|_1 &\le& C \|\vecu\|, \\
\|\vecw\|_{s+1} &\le& C \|\mcurl_{\vecbeta} \vecw\|_s, \\
\|\phi\|_{s+1} &\le& C \|\nabla_{\vecbeta} \phi\|_s.
\end{eqnarray*}
\end{lemma}

\section{The mixed formulation}\label{mixed}
In this section, we first formulate equations (\ref{MengMonk2016MaxwellCell1}) -- (\ref{MengMonk2016MaxwellCell2}) using a mixed formulation. In practice it is often desired to compute the Bloch variety along specific directions in the first Brillouin zone.  So we assume that $\veck=\vecalpha_0+\lambda\hat\vecalpha$ where $\hat{\vecalpha}$ is a fixed unit wave vector and $\vecalpha_0$ is assumed to belong to the first Brillouin zone $[-\pi,\pi]^3$. Then to regularize the problem we introduce
parameters $\tau>0$ and $\eta$ such that $\lambda=\eta+\tau$ so that
\[
\veck=(\vecalpha_0+\tau\hat\vecalpha)+\eta\hat\vecalpha
\]
and denote by $\vecbeta=(\vecalpha_0+\tau\hat\vecalpha)$ the regularization vector. We assume that $\vecbeta$ belongs to the first Brillouin zone $[-\pi,\pi]^3$. For a fixed parameter $\tau$, we aim to compute $\eta$, and hence $\lambda$. 

To derive the mixed formulation, we multiply equation \eqref{MengMonk2016MaxwellCell1} by $\overline{\vecv}$ and integrate by parts
\begin{eqnarray*}
&&\int_\Omega \big( \mcurl + i \veck \times \big) \Big( \frac{1}{\epsilon}  \big(\mcurl \vecu+ i \veck \times \vecu \big) \Big) \cdot \overline{\vecv} dx -\omega^2 \int_\Omega \vecu \cdot \overline{\vecv} dx \\
&=& \int_\Omega   \frac{1}{\epsilon}  \big(\mcurl \vecu+ i \veck \times \vecu \big)  \cdot \big( \mcurl \overline{\vecv} -i \veck \times \overline{\vecv} \big) dx  -\omega^2 \int_\Omega \vecu \cdot \overline{\vecv} dx.
\end{eqnarray*}
Since $\veck$ can be complex-valued, $\mcurl \overline{\vecv} -i \veck \times \overline{\vecv}  = \overline{\mcurl_{\overline{\veck}} \vecv}$.
Now if $\veck=\vecbeta + \eta \hat{\vecalpha}$ where $\vecbeta$ is real-valued, a direct calculation yields
\begin{eqnarray}
&&\int_\Omega \big( \mcurl + i \veck \times \big) \Big( \frac{1}{\epsilon}  \big(\mcurl \vecu+ i \veck \times \vecu \big) \Big) \cdot \overline{\vecv}\, dx -\omega^2 \int_\Omega \vecu \cdot \overline{\vecv} \, dx \nonumber  \\
&=& \left( \epsilon^{-1} \mcurl_{\svecbeta} \vecu, \mcurl_{\svecbeta} \vecv \right) + \eta \left( \epsilon^{-1} i \hat{\vecalpha} \times \vecu, \mcurl_{\svecbeta} \vecv \right) + \eta \left( \epsilon^{-1} \mcurl_{\svecbeta} \vecu, i \hat{\vecalpha} \times \vecv \right)\nonumber  \\
&& - \omega^2 \left( \vecu, \vecv\right) + \eta^2 \left( \epsilon^{-1} i \hat{\vecalpha} \times \vecu,  i \hat{\vecalpha} \times \vecv \right), \label{mixed calculation 1}
\end{eqnarray}
where, for any suitable functions $\vecf$ and $\vecg$, {we define}
\[
(\vecf,\vecg)=\int_{\Omega}\vecf\cdot\overline{\vecg}\,dx.
\]
Note that $\vecu$ in addition satisfies  condition \eqref{MengMonk2016MaxwellCell2}, so that for any $q \in H^1_p(\Omega)$
\begin{eqnarray*}
0&=&\int_\Omega \big((\nabla + i\veck)\cdot \vecu \big) ~~ \overline{q} dx =  -\int_\Omega \vecu \cdot \big(  \nabla \overline{q} -i\veck \overline{q} \big) dx,
\end{eqnarray*}
and since $\veck=\vecbeta + \eta \hat{\vecalpha}$ and $\vecbeta$ is real-valued, we can rewrite this as
\begin{eqnarray}
 - \overline{\big( \nabla_{\vecbeta} q, \vecu \big)}- \eta \overline{\big( i \hat{\vecalpha} q, \vecu \big)}&=&0. \label{mixed calculation 2}
\end{eqnarray}
Now let us introduce a stable mixed formulation {using} \eqref{mixed calculation 1} -- \eqref{mixed calculation 2}. Let $H(\C)$ denote the space consisting of constant functions
on $\Omega$. We impose the additional constraint that $p$ may be chosen so that
\begin{equation}
\int_\Omega p\,dx=0\label{av0}.
\end{equation}
We can now introduce Lagrange multipliers to enforce (\ref{mixed calculation 2}) and (\ref{av0}).  We arrive at the 
the following problem: find {non-trivial} $(\vecu, p, s) \in \vecH_p (\mcurl ; \Omega) \times H^1_{p} (\Omega) \times H(\C)$ and $\eta \in \C$ such that
\begin{eqnarray}
&&\left( \epsilon^{-1} \mcurl_{\svecbeta} \vecu, \mcurl_{\svecbeta} \vecv \right) + \eta \left( \epsilon^{-1} i \hat{\vecalpha} \times \vecu, \mcurl_{\svecbeta} \vecv \right) \nonumber\\&&+ \eta \left( \epsilon^{-1} \mcurl_{\svecbeta} \vecu, i \hat{\vecalpha} \times \vecv \right) - \omega^2 \left( \vecu, \vecv\right) \label{improved mixed 1}  \\
&&+ \eta^2 \left( \epsilon^{-1} i \hat{\vecalpha} \times \vecu,  i \hat{\vecalpha} \times \vecv \right) + \left( \nabla_{\svecbeta} \, p, \vecv \right) + \eta \left(  i \hat{\vecalpha} p, \vecv  \right) + (p,t)= 0, \nonumber\\
&& \overline{\left( \nabla_{\svecbeta}\, q, \vecu \right) } + \overline{(q,s)} + \eta \overline{ \left(  i \hat{\vecalpha} q, \vecu  \right)} = 0, \label{improved mixed 2} 
\end{eqnarray}
for all $ \vecv \in \vecH_p (\mcurl ; \Omega), $ $q \in H^1_p(\Omega),$ and  $ t \in H(\C)$, where the terms $\left( \nabla_{\svecbeta} \, p, \vecv \right) + \eta \left(  i \hat{\vecalpha} p, \vecv  \right)$ and $(p,t)$ serve {to
define Lagrange multipliers } and result in a mixed variational formulation. 
\begin{remark}
It is necessary to introduce the {additional Lagrange multiplier $s$}. If \eqref{improved mixed 1}--\eqref{improved mixed 2} does not include the terms $(p,t)$ and $\overline{(q,s)}$, then $p$ can be any constant in the case when $\veck=0$; we have found that the resulting mixed formulation is then numerically unstable. 
\end{remark}
We can then easily prove the equivalence of the above mixed problem with the {original} Maxwell problem: 
\begin{lemma} \label{pde to improved mixed 1}
Assume that  $\omega\not=0$, $\Re \veck \in [-\pi,\pi]^3$, and $\Im \veck \in [-\pi,\pi]^3$. If $\vecu$ is a solution to \eqref{MengMonk2016MaxwellCell1}--\eqref{MengMonk2016MaxwellCell2}, then there exists $(p,s) \in H^1_p(\Omega) \times H(\C)$ such that $(\vecu,p,s)$ and $\eta$ satisfy the quadratic eigenvalue problem \eqref{improved mixed 1}--\eqref{improved mixed 2}. The converse statement also holds.
\end{lemma}
\begin{proof}
First suppose $\vecu$ is a solution to \eqref{MengMonk2016MaxwellCell1}--\eqref{MengMonk2016MaxwellCell2}, then from equation \eqref{mixed calculation 1}--\eqref{mixed calculation 2}, one can see that $(\vecu, 0,0) \in \vecH_p (\mcurl ; \Omega) \times H^1_{p} (\Omega) \times H(\C)$ and $\eta$ satisfy the quadratic eigenvalue problem \eqref{improved mixed 1}--\eqref{improved mixed 2}.

On the other hand suppose there exists $(p,s) \in H^1_p(\Omega) \times H(\C)$ such that $(\vecu,p,s)$ and $\eta$ satisfy the quadratic eigenvalue problem \eqref{improved mixed 1}--\eqref{improved mixed 2}, then the following holds in the distributional sense,
\begin{eqnarray} 
\mcurl_{\veck} \left( \frac{1}{\epsilon} \mcurl_{\veck} \vecu \right) - \omega^2 \vecu + \nabla_{\veck} p&=&0 \quad \mbox{in} \quad \Omega, \label{improved mixed equiv proof 1}\\
\mdiv_{\veck} \vecu &=& s \quad \mbox{in} \quad \Omega, \label{improved mixed equiv proof 2} \\
(p,t) &=&0 \quad \mbox{for any} \quad t \in H(\C). \label{improved mixed equiv proof 3}
\end{eqnarray}
Now applying $\mdiv_{\veck}$ to \eqref{improved mixed equiv proof 1} and noting \eqref{improved mixed equiv proof 2}
\begin{eqnarray*} 
\mdiv_{\veck} \nabla_{\veck} p=\omega^2 s &\quad \mbox{in} \quad& \Omega.
\end{eqnarray*}
To show that $\vecu$ is a solution to \eqref{MengMonk2016MaxwellCell1}--\eqref{MengMonk2016MaxwellCell2}, it remains to show that $p=0$ and $s=0$. In fact suppose that $p \in H^1_p(\Omega)$ has the following Fourier expansion
\begin{eqnarray*}
 p = \sum_{\vecI \in J} e^{i \vecI \cdot \vecx} p_{\vecI},
\end{eqnarray*}
then 
\begin{eqnarray} 
\omega^2 s=\mdiv_{\veck} \nabla_{\veck} p=  \sum_{\vecI \in J} e^{i \vecI \cdot \vecx} p_{\vecI}  (i\vecI + i\veck)\cdot  (i\vecI + i\veck)\label{spvanish}
\end{eqnarray}
holds in the distributional sense.  {We now show that this implies that $p$ and $s$ vanish:}
\begin{eqnarray*} 
\omega^2 s=  p_0 (i\veck)\cdot (i\veck), \quad \mbox{and} \quad p_{\vecI} (i\vecI + i\veck)\cdot (i\vecI + i\veck)=0 \quad \forall \vecI \not=0.
\end{eqnarray*}
\begin{enumerate}
\item[(a)] Note that $\Omega$ is the unit cell and $\re (\veck)$ belongs to the first Brillouin zone, then
 $(i\vecI + i\veck) \not =0$ for all $\vecI \in J$ and $\vecI \not=0$. Even {if} $(i\vecI + i\veck) \not =0$, {it is possible that} $(i\vecI + i\veck)\cdot (i\vecI + i\veck)$ might be zero since $\veck$ might be complex-valued. Since we restrict that $\Im \veck \in [-\pi,\pi]^3$, then we can show $(i\vecI + i\veck)\cdot (i\vecI + i\veck) \not=0$. Indeed, let $I+\veck =(r_1+i s_1,r_2+i s_2,r_3+i s_3)$, then
\begin{eqnarray*}
-(i\vecI + i\veck)\cdot (i\vecI + i\veck)&=& (r_1^2 + r_2^2+r_3^2 - s_1^2-s_1^2-s_1^2) \\
&&+ 2i (r_1s_1+r_2s_2+r_3s_3).
\end{eqnarray*}
Assume that $(i\vecI + i\veck)\cdot (i\vecI + i\veck)=0$, we show that this is a contradiction. First note that $(i\vecI + i\veck)\cdot (i\vecI + i\veck)=0$ gives
\begin{eqnarray*}
r_1^2 + r_2^2+r_3^2 = s_1^2+s_2^2+s_3^2, \quad r_1s_1+r_2s_2+r_3s_3=0.
\end{eqnarray*}
Since $r_j = i_j + \Re{\veck}_j$ for $j=1,2,3$ and $\Re{\veck} \in [-\pi,\pi]^3$, then 
\begin{eqnarray*}
r_1^2 + r_2^2+r_3^2 \ge 3 \pi^2.
\end{eqnarray*}
Since $\Im \veck \in [-\pi,\pi]^3$, then 
\begin{eqnarray*}
s_1^2+s_2^2+s_3^2 \le 3 \pi^2,
\end{eqnarray*}
and thereby $r_1^2 + r_2^2+r_3^2 = s_1^2+s_2^2+s_3^2$ holds only when 
\begin{eqnarray*}
r_1^2 + r_2^2+r_3^2 = s_1^2+s_2^2+s_3^2= 3 \pi^2, 
\end{eqnarray*} 
where the equations hold when $r_j=\pm \pi$ and $s_j=\pm \pi$ for $j=1,2,3$. However in this case $r_1s_1+r_2s_2+r_3s_3$ cannot be zero and this is a contradiction.

Now {since} $(i\vecI + i\veck)\cdot (i\vecI + i\veck) \not=0$, {equation (\ref{spvanish}) implies that}  $p_{\vecI}=0$ for all $\vecI\not=0$. Thus $p$ is a constant.
\item[(b)] Equation \eqref{improved mixed equiv proof 3} further implies that the constant $p$ has to be zero and hence $s=0$.
\end{enumerate} 
This proves the lemma.
\end{proof}
In order to compute the Bloch variety $(\veck, \omega)$, we first choose a fixed $\omega$, then we compute $\eta$ for a fixed unit wave-vector $\hat{\vecalpha}$ and a fixed regularization wave-vector $\vecbeta$.  Here let us remark that from the eigenvalue problem one can derive {$p=0$} and $s=0$ as in the above proof. In this sense it recovers the mixed formulation in \cite{dobson2001analysis, boffi2006interpolation}. 
\subsection{A linear eigenvalue problem}\label{linear}
For convenience let us denote by
\begin{eqnarray*}
\vecX(\Omega) = \vecH_p (\mcurl ; \Omega) \times \vecL^2 (\Omega).
\end{eqnarray*}
We now {obtain a linear eigenvalue problem from  the quadratic} problem \eqref{improved mixed 1}--\eqref{improved mixed 2}. In this regard we introduce {an auxilliary} function $\vecu_2 = \eta \vecu$. At the same time we {define} $\vecu_1= \vecu$ and denote $\vecU= (\vecu_1,\vecu_2)$. Then the quadratic eigenvalue problem \eqref{improved mixed 1}--\eqref{improved mixed 2} reduces to a   linear eigenvalue problem: find $(\vecU, p,s) \in \vecX(\Omega) \times H^1_{p} (\Omega) \times H(\C)$ and $\eta \in \C$ such that
\begin{eqnarray}
\lefteqn{\left( \epsilon^{-1} \mcurl_{\svecbeta} \vecu_1, \mcurl_{\svecbeta} \vecv_1 \right) + \left( \epsilon^{-1} i \hat{\vecalpha} \times \vecu_2, \mcurl_{\svecbeta} \vecv_1 \right)}\nonumber\\&& + \eta \left( \epsilon^{-1} \mcurl_{\svecbeta} \vecu_1, i \hat{\vecalpha} \times \vecv_1 \right) - \omega^2 \left( \vecu_1, \vecv_1\right)\label{MengMonk2016MaxwellMixedQuadratic1}  \\
&&+ \eta \left( \epsilon^{-1} i \hat{\vecalpha} \times \vecu_2,  i \hat{\vecalpha} \times \vecv_1 \right) + \left( \nabla_{\svecbeta} p, \vecv_1 \right) + \eta \left(  i \hat{\vecalpha} p, \vecv_1  \right) +(p,t)= 0, \nonumber \\
&& \overline{\left( \nabla_{\svecbeta} q, \vecu_1 \right) } +\overline{(q,s)}+ \eta \overline{ \left(  i \hat{\vecalpha} q, \vecu_1  \right)} = 0. \label{MengMonk2016MaxwellMixedQuadratic2} 
\end{eqnarray}
for all {$ (\vecV,q,t) \in \vecX(\Omega) \times H^1_p(\Omega)\times H(\C)$}.\\

{For convenience we now} introduce the following sesquilinear forms. Let 
\begin{eqnarray*}
a_{1}(\vecU; \vecV) &=& \left( \epsilon^{-1} \mcurl_{\svecbeta} \vecu_1, \mcurl_{\svecbeta} \vecv_1 \right) + \left( \epsilon^{-1} i \hat{\vecalpha} \times \vecu_2, \mcurl_{\svecbeta} \vecv_1 \right)  \\
&&- \omega^2 \left( \vecu_1, \vecv_1\right) + M \left( \vecu_2, \vecv_2 \right) , \\
a_{2} (\vecU; \vecV)  &=& \left( \epsilon^{-1} i \hat{\vecalpha} \times \vecu_2,  i \hat{\vecalpha} \times \vecv_1 \right) - M \left( \vecu_1, \vecv_2 \right) \\&&\qquad+ \left( \epsilon^{-1} \mcurl_{\svecbeta} \vecu_1, i \hat{\vecalpha} \times \vecv_1 \right), \\
b_{1}(p; \vecV)  &=&\left( \nabla_{\svecbeta} p, \vecv_1 \right), \\
b_{2}(p; \vecV)  &=&  \left(  i \hat{\vecalpha} p, \vecv_1  \right), \\
c_{1} (s; q) &=& \left( q,s \right),
\end{eqnarray*}
where  $M>0$ is a constant. Here we remark that the sesquilinear forms all depend upon $\vecbeta$, we omit the sub-script $\vecbeta$ as it is clear throughout the paper. 

The linear eigenvalue problem (\ref{MengMonk2016MaxwellMixedQuadratic1})--(\ref{MengMonk2016MaxwellMixedQuadratic2}) then conveniently reads: find {non-trivial} $(\vecU,p,s) \in \vecX(\Omega)\times H^1_{p} (\Omega)\times H(\C)$ and $\eta \in \C$ such that
\begin{eqnarray}
a_{1}(\vecU; \vecV) + b_{1}(p; \vecV) + c_1(t;p) &=& -\eta \left( a_{2}(\vecU; \vecV) + b_{2}(p; \vecV)\right), \label{MengMonk2016MaxwellMixedQuadratic11} \\
 \overline{b_{1}(q; \vecU)} + \overline{c_{1}(s;q) } &=& - \eta \,\overline{ b_{2}(q; \vecU)  }, \label{MengMonk2016MaxwellMixedQuadratic12} 
\end{eqnarray}
for all $(\vecV,q,t) \in  {\vecX(\Omega)\times H^1_{p} (\Omega)\times H(\C)}$. The next lemma verifies 
our claim that this system is equivalent to the original problem.
\begin{lemma}
The quadratic eigenvalue problem \eqref{improved mixed 1}--\eqref{improved mixed 2}  is equivalent to the linear eigenvalue problem (\ref{MengMonk2016MaxwellMixedQuadratic11})--(\ref{MengMonk2016MaxwellMixedQuadratic12}) .
\end{lemma}
\begin{proof}
It is sufficient to show that  the linear eigenvalue problem (\ref{MengMonk2016MaxwellMixedQuadratic11}) -- (\ref{MengMonk2016MaxwellMixedQuadratic12}) yields the quadratic eigenvalue problem \eqref{improved mixed 1}--\eqref{improved mixed 2}. Indeed taking the test function $\vecV=(0,\vecv_2)$ and $t=0$ yield that
\begin{eqnarray*}
M \left( \vecu_2, \vecv_2 \right)  - \eta M \left( \vecu_1, \vecv_2 \right)  =0.
\end{eqnarray*}
This shows that $\vecu_2=\eta \vecu_1$. Plugging $\vecu_2=\eta \vecu_1$ into (\ref{MengMonk2016MaxwellMixedQuadratic11}) -- (\ref{MengMonk2016MaxwellMixedQuadratic12}) yields the quadratic eigenvalue problem \eqref{improved mixed 1}--\eqref{improved mixed 2}.
\end{proof}
\section{Analysis of the linear eigenvalue problem}
Our goal is now to show that linear eigenvalue problem {(\ref{MengMonk2016MaxwellMixedQuadratic11})--(\ref{MengMonk2016MaxwellMixedQuadratic12})} {is equivalent to}  an eigenvalue problem for a compact operator and is thus
appropriate for numerical analysis. We start by introducing the following source problem: find $(\vecU,p,s) \in  \vecX(\Omega) \times H^1_{p} (\Omega)\times H(\C) $ and $\eta \in \C$ such that 
\begin{eqnarray}
a_{1}(\vecU; \vecV) + b_{1}(p; \vecV) + c_{1}(t,p) &=& a_{2}(\vecF; \vecV) + b_{2}(g; \vecV), \label{MengMonk2016MaxwellMixedQuadratic21} \\
 \overline{b_{1}(q; \vecU) } +  \overline{c_{1}(s;q) } &=&  \overline{ b_{2}(q; \vecF)  }, \label{MengMonk2016MaxwellMixedQuadratic22} 
\end{eqnarray}
{for all $(\vecV,q,t) \in  \vecX(\Omega)\times H^1_{p} (\Omega)\times H(\C)$}
where $(\vecF,g) \in  \vecX(\Omega)\times L^2(\Omega)$ {are given functions}. For convenience let us introduce the kernel space
\begin{eqnarray*}
\vecK_{\svecbeta} = \{ \vecV \in   \vecX(\Omega): \quad b_{1}(p; \vecV)=0, \quad \forall \quad p \in H^1_{p} (\Omega) \}.
\end{eqnarray*}
It is readily seen that the kernel $\vecK_{\svecbeta}$ consists of $(\vecv_1,\vecv_2) \in   \vecX(\Omega)$ such that $\nabla_{\svecbeta} \cdot \vecv_1 =0$.

{We now show that the sesquilinear form $a_1(\cdot,\cdot)$ is coercive on $\vecK_{\svecbeta} $:}
\begin{lemma} \label{MengMonk2016MaxwellMixedQuadraticCovercivity}
Let $\vecbeta$ be choosen such that $ \inf_{\vecx\in\Omega}\, |\Re{\epsilon^{-1}(\vecx)}| |\vecgamma^{\vecI}|^2 - \omega^2 > 0$ for any $\vecI \in J$. There exists a sufficiently large $M>0$ such that $a_{1}(\vecU; \vecV)$ satisfies the coercivity condition on $\vecK_{\svecbeta}$, i.e. for any $\vecV \in \vecK_{\svecbeta}$
\begin{eqnarray*}
\Re{ a_{1}(\vecV; \vecV) }\ge C \|\vecV\|^2_{ \vecX(\Omega)},
\end{eqnarray*}
where $C$ is a constant.
\end{lemma}
\begin{proof}
From Young's inequality, 
\begin{eqnarray*}
&&|\left( \epsilon^{-1} i \hat{\vecalpha} \times \vecv_2, \mcurl_{\svecbeta} \vecv_1 \right)  | \\
&\le& c_{\svecbeta} \sup \, |\epsilon^{-1}| \left( \frac{\tau}{2} \|\vecv_2\|^2_{\vecL^2 (\Omega)} + \frac{1}{2\tau} \| \mcurl_{\svecbeta} \vecv_1 \|^2_{\vecL^2(\Omega)} \right),
\end{eqnarray*}
where $c_{\svecbeta}$ is a constant depending on $\vecbeta$, $\tau>0$ is sufficiently large and is to be determined. Now
\begin{eqnarray} \label{coer ineq 1}
\Re{a_{1}(\vecV; \vecV)} &\ge & \left( \Re{\epsilon^{-1}} \mcurl_{\svecbeta} \vecv_1, \mcurl_{\svecbeta} \vecv_1 \right) \nonumber\\&&- c_{\svecbeta} \sup |\epsilon^{-1}| \left( \frac{\tau}{2} \|\vecv_2\|^2_{\vecL^2 (\Omega)} + \frac{1}{2\tau} \|\vecv_1\|^2_{\vecH_p (\mcurl ; \Omega)} \right) \nonumber \\
&&- \omega^2 \left( \vecv_1, \vecv_1\right) + M \left( \vecv_2, \vecv_2 \right) .
\end{eqnarray}
From Lemma \ref{MengMonk2016DecompositionVectorPeriodic} we can have the following decomposition, 
\begin{eqnarray*}
\vecv_1 = \mcurl_{\svecbeta} \vecw_{\vecv_1} + \nabla_{\svecbeta} \phi_{\vecv_1},
\end{eqnarray*}
where
\begin{eqnarray*}
\vecw_{\vecv_1}  = \sum_{\vecI \in J} e^{i \vecI\cdot \vecx} \vecC_{\vecI}.
\end{eqnarray*}
Since $\vecV \in \vecK_{\svecbeta}$, then $\nabla_{\svecbeta} \cdot  \nabla_{\svecbeta} \phi_{\vecv_1} =0$ and consequently $\nabla_{\svecbeta} \phi_{\vecv_1}=0$ and $\vecv_1 = \mcurl_{\svecbeta} \vecw_{\vecv_1}$. Now one can write out explicitly 
\begin{eqnarray}\label{coer ineq 2}
\|\vecv_1\|^2_{\vecH_p (\mcurl ; \Omega)} = \sum_{\vecI \in J} |N_{\vecI} (N_{\vecI}\vecC_{\vecI})|^2 + \sum_{\vecI \in J} |N_{\vecI}\vecC_{\vecI}|^2.
\end{eqnarray}
From equation \eqref{MLMMaxwellDecompositionEqn} one has 
\begin{eqnarray} \label{coer ineq 3}
\lefteqn{\|\mcurl_{\svecbeta} \vecv_1\|^2 = \sum_{\vecI \in J} |N_{\vecI} (N_{\vecI}\vecC_{\vecI})|^2}\nonumber\\&=& (1-\delta) \sum_{\vecI \in J} |N_{\vecI} (N_{\vecI}\vecC_{\vecI})|^2 + \delta \sum_{\vecI \in J} |\vecgamma^{\vecI}|^2|(N_{\vecI}\vecC_{\vecI})|^2,
\end{eqnarray}
where $\delta$ is a constant to be determined. {Substituting} \eqref{coer ineq 2}--\eqref{coer ineq 3} into \eqref{coer ineq 1} one {obtains}
\begin{eqnarray*}
\Re{a_{1}(\vecV; \vecV)} &\ge& \left( (1-\delta ) \inf |\Re{\epsilon^{-1}}| - \frac{c_{\svecbeta}}{2\tau}  \sup  |\epsilon^{-1}| \right) \sum_{\vecI \in J} |N_{\vecI} (N_{\vecI}\vecC_{\vecI})|^2  \\
&&~~~ + \sum_{\vecI \in J} |N_{\vecI}\vecC_{\vecI}|^2 \left( \delta \,\inf \, |\Re{\epsilon^{-1}}| |\vecgamma_{\vecI}|^2 - \omega^2- \frac{c_{\svecbeta}}{2\tau}  \sup  |\epsilon^{-1}|  \right) \\
&& \, + (M- \frac{c_{\svecbeta} \tau}{2} \sup  |\epsilon^{-1}|)  \|\vecv_2\|^2_{\vecL^2 (\Omega)}.
\end{eqnarray*}
From the {assumptions of the lemma}, $\vecbeta$ is chosen such that $ \inf \, |\Re{\epsilon^{-1}}| |\vecgamma_{\vecI}|^2 - \omega^2 > 0$ for any $\vecI \in J$. This is possible under some assumptions made on $\epsilon$, see the following Remark \ref{assumption epsilon} for more details. Let $\delta \in (0,1)$ be sufficiently close to $1$ such that {$\delta \,\inf |\Re{\epsilon^{-1}}| |  \gamma_I|^2 - \omega^2 > 0$}. Let $\tau>0$ be sufficiently large such that $ (1-\delta ) \inf |\Re{\epsilon^{-1}}| - \frac{c_{\svecbeta}}{2\tau}  \sup  |\epsilon^{-1}| > 0$ and $ \delta \,\inf \, |\Re{\epsilon^{-1}}| |\vecgamma_{\vecI}|^2 - \omega^2- \frac{c_{\svecbeta}}{2\tau}  \sup  |\epsilon^{-1}| > 0$. Finally let $M>0$ be sufficiently large such that $M- \frac{c_{\svecbeta} \tau}{2} \sup  |\epsilon^{-1}|>0$. This shows that there exists a constant $C$ such that
\begin{eqnarray*}
\Re{ a_{1}(\vecV; \vecV) } &\ge& C  \big( \sum_{\vecI \in J} |N_{\vecI} (N_{\vecI}\vecC_{\vecI})|^2+ \sum_{\vecI \in J} |N_{\vecI}\vecC_{\vecI}|^2 +  \|\vecv_2\|^2_{\vecL^2 (\Omega)}\big) \\
&\ge& C \|\vecV\|^2_{\vecX(\Omega)}.
\end{eqnarray*}
This proves the lemma.
\end{proof} 

\begin{remark} \label{assumption epsilon}
In Lemma \ref{MengMonk2016MaxwellMixedQuadraticCovercivity}, $\vecbeta$ is chosen such that $ \inf\, |\Re{\epsilon^{-1}}| |\vecgamma^{\vecI}|^2 - \omega^2 > 0$ for any $\vecI \in J$. Here we give a sufficient condition on $\epsilon$ such that the existence of $\vecbeta$ is guaranteed. Recall that
\begin{eqnarray*}
\vecgamma^{\vecI} = \vecI + \vecbeta,\mbox{ where }\vecbeta=\vecalpha_0+\tau\hat\vecalpha.
\end{eqnarray*}
Here $\vecalpha_0\in [-\pi,\pi]^3$ and $\hat{\vecalpha}=(\hat{\alpha}_1,\hat{\alpha}_2,\hat{\alpha}_3)$ is a unit vector. Let $\vecI=2\pi(i_1,i_2,i_3)$ and $\tilde{\vecI}=\vecI+\vecalpha_0=2\pi(\tilde{i}_1,\tilde{i}_2,\tilde{i}_3)$, we have
\begin{eqnarray*}
|\vecgamma^{\vecI}|^2 &=& |(2\pi \tilde{i}_1 + \tau \hat{\alpha}_1,2\pi \tilde{i}_2 + \tau \hat{\alpha}_2,2\pi \tilde{i}_3 + \tau \hat{\alpha}_3)| \\
&=& 4\pi^2 (\tilde{i}_1^2+\tilde{i}_2^2+\tilde{i}_3^2) + \tau^2 + 4\pi \tau (\tilde{i}_1 \hat{\alpha}_1 + \tilde{i}_2 \hat{\alpha}_2+\tilde{i}_3 \hat{\alpha}_3) \\
&\ge &  4\pi^2 (\tilde{i}_1^2+\tilde{i}_2^2+\tilde{i}_3^2) + \tau^2  -  4\pi |\tau|\sqrt{\tilde{i}_1^2+\tilde{i}_2^2+\tilde{i}_3^2} \\
&= & (2\pi \sqrt{\tilde{i}_1^2+\tilde{i}_2^2+\tilde{i}_3^2} -|\tau| )^2.
\end{eqnarray*}
We discuss the following three cases.
\begin{enumerate}
\item[(a)]
The first case is $\vecalpha_0=0$.  In this case for all $\vecI=2\pi(i_1,i_2,i_3)$
\begin{eqnarray*}
|\vecgamma^{\vecI}|^2  
&\ge & (2\pi \sqrt{i_1^2+i_2^2+i_3^2} -|\tau| )^2 \\
& \ge & \min \{(2\pi-|\tau|)^2,\tau^2\}.
\end{eqnarray*}
Then to guarantee $ \inf\, |\Re{\epsilon^{-1}}| |\vecgamma^{\vecI}|^2 - \omega^2 > 0$ for any $\vecI \in J$, it is sufficient to have
\begin{eqnarray*}
\omega^2 \le \min \{(2 \pi -|\tau|)^2, \tau^2\} \inf |\Re{\epsilon^{-1}}|.
\end{eqnarray*}
For instance one can choose $\tau=\pi$, then the real part of $\frac{1}{\epsilon}$ can not be too small in order to compute $\omega$ in a certain range. A similar choice for the scalar case has been discussed in \cite{engstrom2014spectral}. 
\item[(b)]
Consider the case  $\vecalpha_0=(\pi,0,0)$, one can check that 
\begin{eqnarray*}
\tilde{\vecI} \in J+ 2\pi\left(\frac{1}{2},0,0\right)
\end{eqnarray*}
where $J=\{ 2\pi(i_1,i_2,i_3): \mbox{for integers} \,\,  i_1, i_2, i_3 \}$, this yields that
\begin{eqnarray*}
2\pi \sqrt{\tilde{i}_1^2+\tilde{i}_2^2+\tilde{i}_3^2} \ge \pi,
\end{eqnarray*}
where equality holds when $\tilde{i}_1=\frac{1}{2}$, and $\tilde{i}_2=\tilde{i}_3=0$. In this case let us pick $\tau=0$, i.e. $\vecbeta=(\pi,0,0)$, this shows that  the coercivity guaranteed by  Lemma \ref{MengMonk2016MaxwellMixedQuadraticCovercivity} holds  when
\begin{eqnarray*}
\omega^2 \in (0, \pi^2 \inf_{\vecx\in\Omega}  |\Re{\epsilon^{-1}(\vecx)}|).
\end{eqnarray*}
\item[(c)]
Consider the case $\vecalpha_0=(\pi,\pi,0)$, one can check that 
\begin{eqnarray*}
\tilde{\vecI} \in J+ 2\pi\left(\frac{1}{2},\frac{1}{2},0\right)
\end{eqnarray*}
where $J=\{ 2\pi(i_1,i_2,i_3): \mbox{for integers} \,\,  i_1, i_2, i_3 \}$, this yields that
\begin{eqnarray*}
2\pi \sqrt{\tilde{i}_1^2+\tilde{i}_2^2+\tilde{i}_3^2} \ge \sqrt{2}\pi,
\end{eqnarray*}
where equality holds when $\tilde{i}_1=\tilde{i}_2=\frac{1}{2}$, and $\tilde{i}_3=0$. In this case we can pick $\tau=0$, i.e. $\vecbeta=(\pi,\pi,0)$, this gives the coercivity guaranteed by Lemma \ref{MengMonk2016MaxwellMixedQuadraticCovercivity} when
\begin{eqnarray*}
\omega^2 \in (0, 2\pi^2 \inf_{\vecx\in\Omega} |\Re{\epsilon^{-1}}(\vecx)|).
\end{eqnarray*}
\end{enumerate}
\end{remark}
{Next we verify that $b_1(\cdot,\cdot)$ satisfies an inf-sup condition:}
\begin{lemma}  \label{MengMonk2016MaxwellMixedQuadraticInfSup}
$b_{1}(p; \vecV)$ satisfies the inf-sup condition, i.e. for any $p \in H^1_{p} (\Omega)$ there exists $\vecV\in \vecX(\Omega)$  such that
\begin{eqnarray*}
b_{1}(p; \vecV) \ge C  \|p\|_1\quad \mbox{and} \quad  \|\vecV\|_{\vecX(\Omega)} \le C  \|p\|_1,
\end{eqnarray*}
where $C$ is a constant  {independent of $p$.}
\end{lemma}
\begin{proof}
For any $p \in H^1_{p} (\Omega)$, let $\vecV = (\nabla_{\svecbeta} p, 0)$. Then from Lemma \ref{MengMonk2016DecompositionVectorPeriodic}
\begin{eqnarray*}
b_{1}(p; \vecV) \ge C \|\nabla_{\svecbeta} p\| \ge C \|p\|_1 \quad \mbox{and} \quad  \|\vecV\|_{\vecX(\Omega)} \le C\|p\|_1,
\end{eqnarray*}
where $C$ is a constant.
This proves the lemma.
\end{proof} 
From Lemma \ref{MengMonk2016MaxwellMixedQuadraticCovercivity} and Lemma \ref{MengMonk2016MaxwellMixedQuadraticInfSup}, one can first solve $(\vecU,p)$ with unknown $s$, then apply $(t,p)=0$ to get $s$. Therefore the source problem \eqref{MengMonk2016MaxwellMixedQuadratic21}--\eqref{MengMonk2016MaxwellMixedQuadratic22} has a unique solution.
\subsection{Regularity properties of the solution operator}
From Lemma \ref{MengMonk2016MaxwellMixedQuadraticCovercivity} and Lemma \ref{MengMonk2016MaxwellMixedQuadraticInfSup}, one can introduce the solution operator $T$ that maps $(\vecF, g) \in \vecX(\Omega) \times L^2(\Omega)$ to {the solution $(\vecU, p,s) \in  \vecX(\Omega)\times H_p^1(\Omega)\times H(\Omega)$ of the source problem \eqref{MengMonk2016MaxwellMixedQuadratic21}--\eqref{MengMonk2016MaxwellMixedQuadratic22}}. {In particular, if $(\vecF,g)=(\vecf_1,\vecf_2,g)$ and $(\vecU,p,s)=(\vecu_1,\vecu_2,p,s)$ solves \eqref{MengMonk2016MaxwellMixedQuadratic21}--\eqref{MengMonk2016MaxwellMixedQuadratic22} then  
\[
T (\vecf_1,\vecf_2, g):=(\vecu_1,\vecu_2,p,s).\]} Now, { {use Lemma~\ref{MengMonk2016DecompositionVectorPeriodic}}} we decompose $\vecu_1$ as
\begin{eqnarray} \label{Decomu1}
\vecu_1 = \mcurl_{\svecbeta} \vecw_{\vecu_1} + \nabla_{\svecbeta} \phi_{\vecu_1}.
\end{eqnarray}
From Lemma \ref{MengMonk2016DecompositionVectorPeriodic} we also have the following decompositions
\begin{eqnarray}
i \hat{\vecalpha} \times \big( \frac{1}{\epsilon} \mcurl_{\svecbeta} \vecf_1 \big)&=& \mcurl_{\svecbeta} \vecw_{\vecf_1} + \nabla_{\svecbeta} \phi_{\vecf_1}, \label{Decomf1}\\
i \hat{\vecalpha} \times \left( \epsilon^{-1} i \hat{\vecalpha} \times \vecf_2 \right) &=& \mcurl_{\svecbeta} \vecw_{\vecf_2} + \nabla_{\svecbeta} \phi_{\vecf_2}, \label{Decomuf2}\\
i \hat{\vecalpha} g  &=& \mcurl_{\svecbeta}\vecw_{g} + \nabla_{\svecbeta} \phi_{g}, \label{Decomg}
\end{eqnarray}
with $\vecw_{\vecf_1} \in \vecH^1_p(\Omega)$, $\vecw_{\vecf_2} \in \vecH^1_p(\Omega)$ and $\vecw_{g}\in \vecH^1_p(\Omega)$. With {this notation}, we have the following lemma.
\begin{lemma}\label{PrioriEstimate}
Let $(\vecU,p) = T (\vecF, g)$ and $s$ be defined as above. Then the following a priori estimates hold
\begin{eqnarray}
\|(\mcurl_{\svecbeta} \vecw_{\vecu_1}, \vecu_2)\|_{\vecX(\Omega)}   &\le&  c\left( \|\vecw_{\vecf_2}\|_{\vecL^2(\Omega)}+\|{\vecf_1}\|_{\vecL^2(\Omega)} + \|\vecw_{\vecf_1}\|_{\vecL^2(\Omega)} \right.\nonumber\\&&\qquad\left.+ \|\vecw_{g}\|_{\vecL^2(\Omega)} \right), \label{MLMMaxwellMixedQuadraticCompactEstimatecurlu1}\\
\|\nabla_{\svecbeta} \phi_{\vecu_1}\|_{\vecH^1_p(\Omega)} &\le& c \big(\|{\vecf_1}\|_{\vecL^2(\Omega)} +\|s\| \big), \label{MLMMaxwellMixedQuadraticCompactEstimatedivu1}\\
\|p\|_{H^1_p(\Omega)} &\le& c \left( \|{\vecf_1}\|_{\vecL^2(\Omega)} +\|\mcurl \vecf_1\|_{\vecL^2(\Omega)} + \|\vecf_2\|_{\vecL^2(\Omega)} \right.\nonumber\\&&\qquad\left.+ \|g\|_{L^2(\Omega)} \right),\label{MLMMaxwellMixedQuadraticCompactEstimatep} \\
|s| &\le& c \left( \|{\vecf_1}\|_{\vecL^2(\Omega)} +\|\mcurl \vecf_1\|_{\vecL^2(\Omega)} + \|\vecf_2\|_{\vecL^2(\Omega)} \right.\nonumber\\&&\qquad\left.+ \|g\|_{L^2(\Omega)} \right),\label{MLMMaxwellMixedQuadraticCompactEstimates}
\end{eqnarray}
where c is a generic constant.
\end{lemma}
\begin{proof}
From equation (\ref{MengMonk2016MaxwellMixedQuadratic22}), one can derive that
\begin{eqnarray*} 
-(\mdiv_{\svecbeta} \vecu_1,q) + (s,q) = -(i \hat{\vecalpha}\cdot \vecf_1,q),
\end{eqnarray*}
{n}ote from the decomposition (\ref{Decomu1}), one can obtain
\begin{eqnarray} \label{MLMMaxwellMixedQuadraticdivu1}
-(\mdiv_{\svecbeta} \nabla_{\svecbeta} \phi_{\vecu_1}  ,q) + (s,q) = -(i \hat{\vecalpha}\cdot \vecf_1,q), \mbox{ for any } q\in H^1_p(\Omega).
\end{eqnarray}
Let $\vecv_1=0$ and $t=0$ in equation (\ref{MengMonk2016MaxwellMixedQuadratic21}), then one can directly obtain 
 \begin{eqnarray} \label{MLM u2 f1}
 \vecu_2 = -\vecf_1.
\end{eqnarray}
Let $\vecv_1 = \nabla_{\svecbeta} \phi_{\vecv_1} \in K_{\svecbeta}^\perp$ and $t=0$ in equation (\ref{MengMonk2016MaxwellMixedQuadratic21}), then with the help of \eqref{MLMMaxwellMixedQuadraticdivu1}
\begin{eqnarray}
-(\mdiv_{\svecbeta}  \nabla_{\svecbeta} \, p, \phi_{\vecv_1} )&=&-\omega^2  (i \hat{\vecalpha} \cdot \vecf_1 +s, \phi_{\vecv_1} ) + \Big( i \hat{\vecalpha} \times \big(\frac{1}{\epsilon} \mcurl_{\svecbeta} \vecf_1\big), \nabla_{\svecbeta}\phi_{\vecv_1}   \Big)\nonumber\\&&+ ( i \hat{\vecalpha} g,\nabla_{\svecbeta}\phi_{\vecv_1}   )  + ~  \Big(  i \hat{\vecalpha} \times \big( \frac{1}{\epsilon}  i \hat{\vecalpha} \times \vecf_2  \big) , \nabla_{\svecbeta}\phi_{\vecv_1}   \Big), \label{MengMonk2016MaxwellMixedQuadratic32}
\end{eqnarray}
for any $\phi_{\vecv_1}\in H^1_p(\Omega)$.
Let  $ \phi_{\vecv_1}=p$ in equation (\ref{MengMonk2016MaxwellMixedQuadratic32}), one can obtain
\begin{eqnarray*}
\|p\|_1 \le c \left( \|\vecf_1\| +\|\mcurl \vecf_1\| + \|\vecf_2\| + \|g\| + |s| \right).
\end{eqnarray*}
Next we estimate $s$. Let $\vecV=0$ in (\ref{MengMonk2016MaxwellMixedQuadratic21}), then one can obtain
\begin{eqnarray*}
(t,p)=0 \mbox{ for any constant } t, \mbox{ i.e. } (1,p)=0.
\end{eqnarray*}
Now let $\phi_{\vecv_1}=s$ in \eqref{MengMonk2016MaxwellMixedQuadratic32}, one can obtain
\begin{eqnarray*}
|s| \le c \left( \|\vecf_1\| +\|\mcurl \vecf_1\| + \|\vecf_2\| + \|g\| \right).
\end{eqnarray*}
{Taking these estimates together yields} estimates (\ref{MLMMaxwellMixedQuadraticCompactEstimatep}) -- \eqref{MLMMaxwellMixedQuadraticCompactEstimates}. 

From equation \eqref{MLMMaxwellMixedQuadraticdivu1}, one can obtain 
\begin{eqnarray*}
\mdiv_{\svecbeta} \nabla_{\svecbeta} \phi_{\vecu_1} = i \hat{\vecalpha} \cdot \vecf_1+s.
\end{eqnarray*}
Since $ \mcurl_{\svecbeta}\nabla_{\svecbeta} \phi_{\vecu_1}=0$, then one can derive that estimate (\ref{MLMMaxwellMixedQuadraticCompactEstimatedivu1}) holds.

Now we derive \eqref{MLMMaxwellMixedQuadraticCompactEstimatecurlu1}. Note that $\tilde{\vecU}=(\mcurl_{\svecbeta} \vecw_{\vecu_1},\vecu_2) \in K_{\svecbeta}$ satisfies the variational form
\begin{eqnarray}
a_{1}(\tilde{\vecU}; \vecV)  &=& a_{2}(\vecF; \vecV) + b_{2}(g; \vecV) \quad \forall \quad \vecV \in  K_{\svecbeta}. \label{MengMonk2016MaxwellMixedQuadraticKtau} 
\end{eqnarray}
To estimate the right hand side of \eqref{MengMonk2016MaxwellMixedQuadraticKtau}, we first observe that
\begin{eqnarray}
&& |a_{2}(\vecF; \vecV) + b_{2}(g; \vecV)| \nonumber \\
& = & | \left( \epsilon^{-1} i \hat{\vecalpha} \times \vecf_2,  i \hat{\vecalpha} \times \vecv_1 \right) - M \left( \vecf_1, \vecv_2 \right)  +\left( \epsilon^{-1} \mcurl_{\svecbeta} \vecf_1, i \hat{\vecalpha} \times \vecv_1 \right) + \left(  i \hat{\vecalpha} g, \vecv_1  \right)| \nonumber \\
&=& |\left( i \hat{\vecalpha} \times \big( \epsilon^{-1} i \hat{\vecalpha} \times \vecf_2 \big),  \vecv_1 \right) - M \left( \vecf_1, \vecv_2 \right) + \left(i \hat{\vecalpha} \times \big( \epsilon^{-1} \mcurl_{\svecbeta} \vecf_1 \big), \vecv_1 \right)  \nonumber \\
&&~+ \left(  i \hat{\vecalpha} g, \vecv_1  \right)|. \label{MLMMaxwellMixedQuadraticataubtau}
\end{eqnarray}
Note that for any $\vecV=(\vecv_1,\vecv_2) \in K_{\svecbeta}$, we have that 
\begin{eqnarray*}
\vecv_1 = \mcurl_{\svecbeta} \vecw_{\vecv_1} \quad \mbox{and} \quad \nabla_{\svecbeta} \cdot \vecw_{\vecv_1} =0.
\end{eqnarray*}
From equations (\ref{Decomf1}) -- (\ref{Decomg}) and integration by parts
\begin{eqnarray}
\left(i \hat{\vecalpha} \times \big( \epsilon^{-1} \mcurl_{\svecbeta} \vecf_1 \big), \vecv_1 \right) &=& \left( \vecw_{\vecf_1} ,  \mcurl_{\svecbeta}  \vecv_1 \right), \label{MLMMaxwellMixedQuadraticf1v1}\\
\left( i \hat{\vecalpha} \times \left( \epsilon^{-1} i \hat{\vecalpha} \times \vecf_2 \right),  \vecv_1 \right) 
&=& \left( \vecw_{\vecf_2} ,  \mcurl_{\svecbeta}  \vecv_1 \right), \label{MLMMaxwellMixedQuadraticf2v1}\\
\left(  i \hat{\vecalpha} g, \vecv_1  \right) &=& \left( \vecw_g, \mcurl_{\svecbeta} \vecv_1  \right).\label{MLMMaxwellMixedQuadraticgv1}
\end{eqnarray}
Now from equations (\ref{MLMMaxwellMixedQuadraticataubtau})--(\ref{MLMMaxwellMixedQuadraticgv1})  we have that
 \begin{eqnarray}
&& \left|a_{2}(\vecF; \vecV) + b_{2}(g; \vecV) \right| \nonumber \\
& \le & c\Big(( \|\vecw_{\vecf_2}\|_{\vecL^2(\Omega)} \|\mcurl_{\svecbeta} \vecv_1\|_{\vecL^2(\Omega)}  + \|\vecf_1\|_{\vecL^2(\Omega)} \| \vecv_2\|_{\vecL^2(\Omega)} +  \|\vecw_{\vecf_1}\|_{\vecL^2(\Omega)} \|\mcurl_{\svecbeta} \vecv_1\|_{\vecL^2(\Omega)} \nonumber \\
&&~~+\|\vecw_{g}\|_{\vecL^2(\Omega)} \|\mcurl_{\svecbeta} \vecv_1\|_{\vecL^2(\Omega)} \Big) \nonumber \\
&\le & c\Big( \|\vecw_{\vecf_2}\|_{\vecL^2(\Omega)} + \|\vecf_1\|_{\vecL^2(\Omega)}+\|\vecw_{\vecf_1}\|_{\vecL^2(\Omega)} + \|\vecw_{g}\|_{\vecL^2(\Omega)}\Big) \|\vecV\|_{\vecX(\Omega)}.
 \label{MLMMaxwellMixedQuadraticataubtauestimate}
\end{eqnarray}
From Lemma \ref{MengMonk2016MaxwellMixedQuadraticCovercivity}, equation (\ref{MengMonk2016MaxwellMixedQuadraticKtau}) and (\ref{MLMMaxwellMixedQuadraticataubtauestimate}), we have that 
\begin{eqnarray*}
\|\tilde{\vecU}\|^2_{\vecX(\Omega)} &\le& c \Re{a_{1}(\tilde{\vecU}; \tilde{\vecU}) } \\ 
&\le & c \left|a_{2}(\vecF; \tilde{\vecU}) + b_{2}(g; \tilde{\vecU}) \right| \\
&\le & c\Big( \|\vecw_{\vecf_2}\|_{\vecL^2(\Omega)} + \|\vecf_1\|_{\vecL^2(\Omega)}+\|\vecw_{\vecf_1}\|_{\vecL^2(\Omega)} + \|\vecw_{g}\|_{\vecL^2(\Omega)}\Big)\|\tilde{\vecU}\|_{\vecX(\Omega)}.
\end{eqnarray*}
This yields that 
\begin{eqnarray*}
\|\tilde{\vecU}\|_{\vecX(\Omega)}   \le  c\Big( \|\vecw_{\vecf_2}\|_{\vecL^2(\Omega)} + \|\vecf_1\|_{\vecL^2(\Omega)}+\|\vecw_{\vecf_1}\|_{\vecL^2(\Omega)} + \|\vecw_{g}\|_{\vecL^2(\Omega)}\Big),
\end{eqnarray*}
i.e. estimate (\ref{MLMMaxwellMixedQuadraticCompactEstimatecurlu1}) holds. This proves the lemma.
\end{proof} 
{Now we prove that $T$ is compact:}
\begin{lemma} \label{MLMRegularity}
The solution operator
\begin{eqnarray*}
T:  \left( \vecH_p^1(\Omega) \times \vecL_p^2(\Omega) \right) \times L_p^2(\Omega) \to  \left( \vecH_p^1(\Omega) \times \vecL^2(\Omega) \right) \times L^2(\Omega){\times H(\Omega)}
\end{eqnarray*}
is compact.
\end{lemma}
\begin{proof}
Now suppose $\big( (\vecf_1)_j,(\vecf_2)_j,g_j\big)$ converges weakly  to zero in $  \vecH_p^1(\Omega) \times \vecL_p^2(\Omega) \times L_p^2(\Omega) $. Let $\big( (\vecu_1)_j,(\vecu_2)_j,p_j{,s_j}\big)=T \big( (\vecf_1)_j,(\vecf_2)_j,g_j \big)$ and 
\begin{eqnarray*} 
(\vecu_1)_j = \mcurl_{\svecbeta} (\vecw_{\vecu_1})_j + \nabla_{\svecbeta} (\phi_{\vecu_1})_j.
\end{eqnarray*}
Furthermore let $s_j$ be such that $\big( (\vecu_1)_j,(\vecu_2)_j,p_j,s_j\big)$ solves the source problem \eqref{MengMonk2016MaxwellMixedQuadratic21}--\eqref{MengMonk2016MaxwellMixedQuadratic22} with source $\big( (\vecf_1)_j,(\vecf_2)_j,g_j \big)$.

Analogous to equations (\ref{Decomf1}) -- (\ref{Decomg}), let
\begin{eqnarray*}
i \hat{\vecalpha} \times \big( \frac{1}{\epsilon} \mcurl_{\svecbeta} (\vecf_1)_j \big)&=& \mcurl_{\svecbeta} (\vecw_{\vecf_1})_j + \nabla_{\svecbeta} (\phi_{\vecf_1})_j, \quad \mdiv_{\svecbeta} (\vecw_{\vecf_1})_j=0,\\
i \hat{\vecalpha} \times \left( \epsilon^{-1} i \hat{\vecalpha} \times (\vecf_2)_j \right) &=& \mcurl_{\svecbeta} (\vecw_{\vecf_2})_j + \nabla_{\svecbeta} (\phi_{\vecf_2})_j, \quad \mdiv_{\svecbeta} (\vecw_{\vecf_2})_j=0, \\
i \hat{\vecalpha} g_j  &=& \mcurl_{\svecbeta} (\vecw_{g})_j + \nabla_{\svecbeta} (\phi_{g})_j,  \quad \mdiv_{\svecbeta} (\vecw_{g})_j=0,
\end{eqnarray*}
then it can be seen that $(\vecw_{\vecf_1})_j$, $(\vecw_{\vecf_2})_j$ and $(\vecw_{g})_j $ converges weakly  in $\vecH^1_p(\Omega)$. The compact embedding from $\vecH^1_p(\Omega)$ to $\vecL^2(\Omega)$ yields that there exists a sequence (still denoted as) $(\vecw_{\vecf_1})_j$, $(\vecw_{\vecf_2})_j$,  $(\vecw_{g})_j$ and $(\vecf_1)_j $ converging strongly  to zero in $\vecL^2(\Omega)$. From equation (\ref{MLMMaxwellMixedQuadraticCompactEstimatecurlu1}), $(\mcurl_{\svecbeta} (\vecw_{\vecu_1})_j, (\vecu_2)_j)$ converges strongly to zero in $\vecX(\Omega)$ and hence in $\vecH_p^1(\Omega) \times \vecL^2(\Omega)$.

From (\ref{MLMMaxwellMixedQuadraticCompactEstimates}), one can obtain that $s_j$ converges weakly  to zero in 
$L^2(\Omega)$. Since $s_j$ are constants, {then $s_j$ strongly converges to zero in $L^2(\Omega)$. In addition, $(\vecf_1)_j$ weakly converges to zero in $\vecH_p^1(\Omega)$, then there exist a sequence (still denoted as)   $(\vecf_1)_j$ strongly converging to zero in $\vecL^2(\Omega)$}. From (\ref{MLMMaxwellMixedQuadraticCompactEstimatedivu1}), $(\nabla_{\svecbeta} \phi_{\vecu_1})_j$ converges strongly to zero in $\vecH^1_p(\Omega)$. 

From equation  (\ref{MLMMaxwellMixedQuadraticCompactEstimatep}), $p_j$ converges weakly to zero in $H^1_p(\Omega)$. The compact embedding from $H^1_p(\Omega)$ to $L^2(\Omega)$ yields that there exist a sequence (still denoted as) {$\{p_j\}$} strongly convergent to zero in $L^2(\Omega)$. 

From equation \eqref{MLM u2 f1}, one can obtain  $(\vecu_2)_j=-(\vecf_1)_j$. Since $(\vecf_1)_j$ converges weakly  to zero in $\vecH_p^1(\Omega)$, then one can obtain that $(\vecf_1)_j$ strongly converges to zero in $\vecL^2(\Omega)$ and therefore $(\vecu_2)_j$ converges strongly to zero in $\vecL^2(\Omega)$.

Hence $T \left( (\vecf_1)_j,(\vecf_2)_j,g_j \right)$ converges strongly to zero in $  \vecH_p^1(\Omega) \times \vecL^2(\Omega) \times L^2(\Omega){\times H(C)}$. This proves the lemma.
\end{proof} 
The following corollary follows immediately from Lemma \ref{MLMRegularity} and compact operator theory. 
\begin{Coro}
For a fixed direction $\hat{\vecalpha}$, a fixed frequency $\omega$ and a fixed regularization wave vector $\vecbeta$, the eigenvalues $\eta$ corresponding to  \eqref{improved mixed 1}--\eqref{improved mixed 2} form at most a discrete set in $\mathbb{C}$.
\end{Coro}
\section{Numerical Analysis}\label{num}
We use a straightforward finite element discretization of the linearized problem (\ref{MengMonk2016MaxwellMixedQuadratic11})--(\ref{MengMonk2016MaxwellMixedQuadratic12}). We use a periodic tetrahedral mesh of $\Omega$ and use $p$-degree edge elements of the second kind~\cite{ned86} to approximate
$\vecH_p (\mcurl ; \Omega)$, and the same space to approximate {the space  $\vecL^2 (\Omega)$ appearing in the definition of $\vecX(\Omega)$}. For the Lagrange multiplier we use $p+1$ degree continuous piecewise linear functions to approximate $H_p^1(\Omega)$. The choice of degree for the scalar space is dictated by~\cite{ned86}.  Because of limitations on memory in our desktop, we have only used $p=1,2,3$ in this paper. We choose $\tau$ as in Remark~\ref{assumption epsilon}. The resulting linear eigenvalue problem is approximated using the Arnoldi method~\cite{arnoldi}. All results were computed using 
Netgen/NGSolve~\cite{schoberl1997netgen} {both to generate the mesh and solve the eigenvalue problems via the NGSpy python interface}. {In practice we set $M=1$.}

As yet we have been unable to prove a convergence rate for the finite element approximation of our
method. In particular, in \cite{boffi2010finite}, the author discusses the analysis of eigenvalue problems using a mixed finite element formulations. Two types of problems are discussed. However the theory does not cover the type of  problem (\ref{MengMonk2016MaxwellMixedQuadratic11})-(\ref{MengMonk2016MaxwellMixedQuadratic12}).
 
We now present two examples.  The first has frequency independent parameters and allows us to validate our code against a more standard finite element method, while the second investigates a problem having a frequency dependetn coefficient.

\subsection*{Example 1:} Our first example uses a frequency independent choice of $\epsilon$.  Hence we can compute the Bloch variety either in the standard way by choosing $\veck$ and computing all relevant $\omega$, or using our new method by fixing $\omega$ and solving the linearized quadratic eigenvalue problem.

We start by using a standard edge element code to compute the Bloch variety via a standard eigenvalue problem, and then compare our results to these calculated eigenvalues. This example is motivated by one of the numerical experiments in \cite{dobson2000efficient} which in turn
is similar to an example in \cite{SH1993}.  The square rod structure for which the unit cell is shown in Fig.~\ref{fig1} (left panel)
consists of rods with $\epsilon=13$ surrounded by air with $\epsilon=1$.  We use the same volume ratio of $0.82$ (ratio between the volume of air and the total volume of the cell) as is used in
\cite{SH1993}.  Since $\epsilon$ is independent 
of $\omega$ we can approximate the eigenvalue problem using periodic edge elements with corresponding 
periodic $H^1$ elements to stabilize the problem (for a similar method see \cite{dobson2001analysis}).  We use first order
edge elements of the second kind and second order vertex elements.  The mesh size requested from the NGSolve mesh generator is 1/3.  We use an Arnoldi scheme to compute approximate eigenvalues.

The wave vector $\veck$ is defined in terms of a parameter $\alpha$ for $0\leq \alpha\leq 3\pi$ as follows
\begin{equation}
\veck=\left\{\begin{array}{ll}(\alpha,0,0)&\mbox{ for }0\leq \alpha \leq \pi\\
(\pi,\alpha-\pi,0)&\mbox{ for }\pi\leq \alpha \leq 2\pi\\
(\pi,\pi,\alpha-2\pi)&\mbox{ for }2\pi\leq \alpha \leq 3\pi\end{array}\right.\label{alphadef}
\end{equation}
For each $\alpha$ we compute the corresponding modes $\omega^2$ and plot the normalized frequency $\omega/(2\pi)$ against $\alpha$.
Our results are shown in Fig.~\ref{fig1} and can be compared to Fig. 8 in \cite{SH1993}.  The band gap is clearly visible, and the qualitative form of the diagram is the same as published work.  
There seems to be a mismatch in units on the y-axis perhaps due to a different scaling in  \cite{SH1993}.

\begin{figure}
\begin{center}
\begin{tabular}{cc}
\resizebox{0.4\textwidth}{!}{\includegraphics{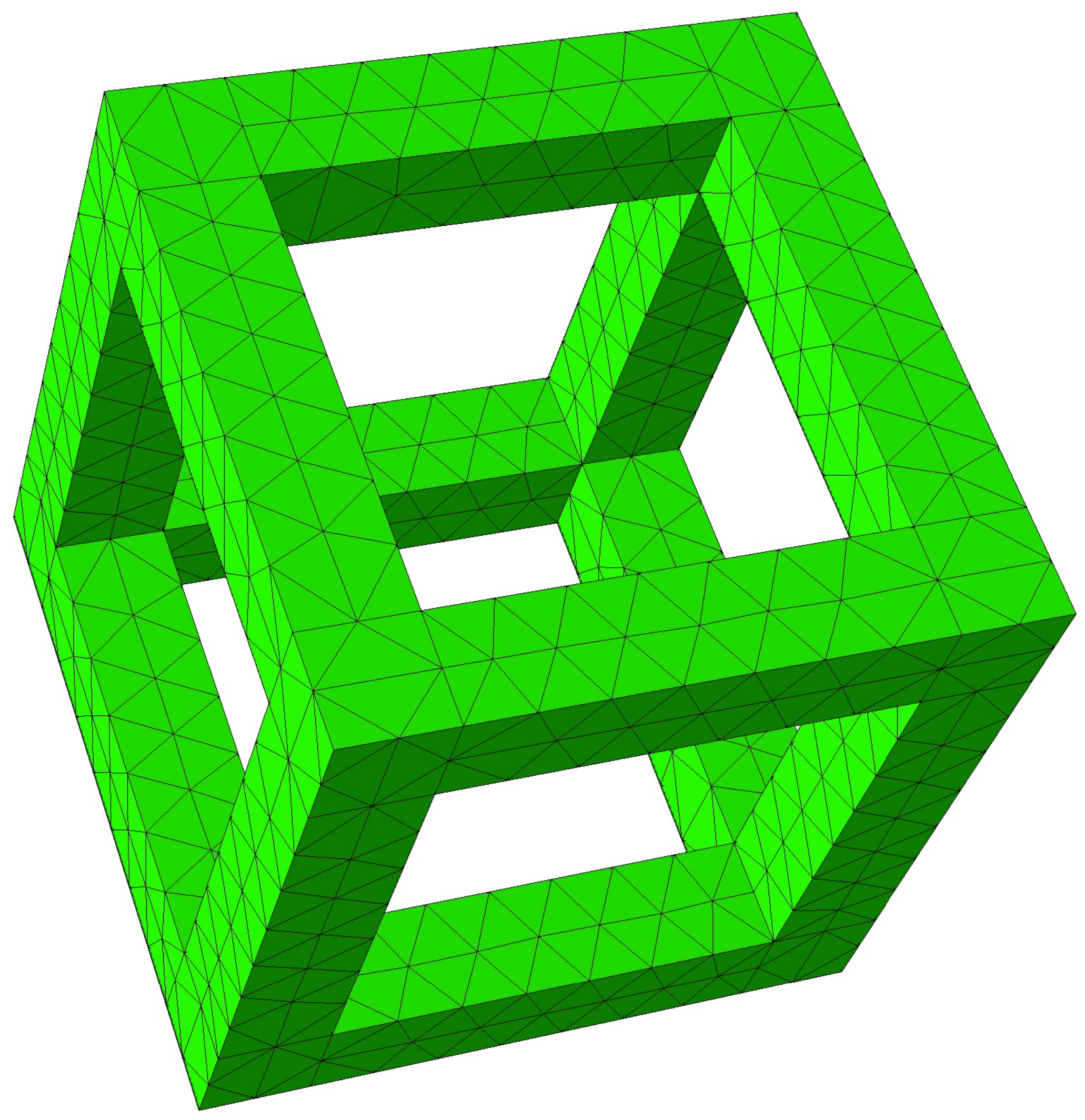}} & \resizebox{0.5\textwidth}{!}{\includegraphics{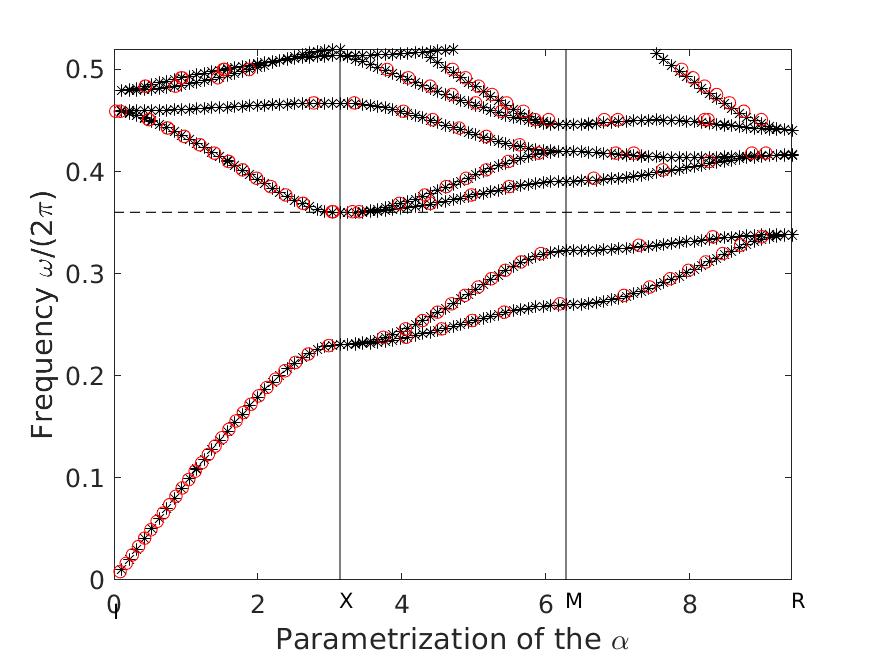}}
\end{tabular}
\end{center}
\caption{Results for the frequency independent choice of $\epsilon$ in Example 1. Left: The mesh in the rods in Example 1 (air is also filled by tetrahedra). Right:  The Bloch variety is computed
in two ways: using a ``standard'' approach by finding $\omega$ as a function of $\veck$ (and hence $\alpha$) shown with ($*$) and using our linearized quadratic eigenvalue solver shown in \textcolor{red}{$\circ$}. As expected there is agreement between the two approaches.}
\label{fig1}
\end{figure}

We then repeat the analysis of this problem using our linearized quadratic eigensolver based on 
(\ref{MengMonk2016MaxwellMixedQuadratic11})--(\ref{MengMonk2016MaxwellMixedQuadratic12}) using our quadratic edge elements and cubic vertex elements.  The mesh is the same as for the standard method discussed above.   Results are shown in Fig.~\ref{fig1}.  Clearly there is good agreement between the two methods {so either method can be used for frequency independent media}.  In practice, because of the larger size of our linearized quadratic eigenvalue problem, the ``standard'' approach is faster if there are no frequency dependent materials present.

\subsection*{Example 2:}  This example is motivated by a study in \cite[Section V]{toader2004photonic} where a frequency
dependent permittivity is considered.  The photonic crystal consists of a face centered cubic lattice of spheres (well known not to support band gaps).  The lattice constant (size of the unit cell) is denoted $a$ and each sphere has radius $r_{cp}=\delta a/(1\sqrt{2})$ where
$\delta\leq 1$ is a constant. Each sphere consists of a central spherical core of radius $0.9\,r_{cp}$ and is covered by a coating of thickness $0.1\,r_{cp}$. In \cite{toader2004photonic} the spheres are close packed so $\delta=1$, and they use
their ``cutting surface'' method together  with an approximation scheme that uses
a  plane wave basis, essentially expanding the fields in terms of a Fourier basis as in
(\ref{FB}).  To simplify mesh generation, we choose $\delta=0.9$.  Thus the results will not be exactly the same as those \cite{toader2004photonic}, but show a similar pattern. 

We start by using a frequency independent choice of $\epsilon$ (setting $\epsilon$ in the coating to that of the inner sphere).  Results are shown in Fig.~\ref{fig2}.  Clearly, as for Example 1, there is good agreement between the 
standard and linearized quadratic approaches.

Moving on to a frequency dependent coating we now use the choice of $\epsilon$ from \cite[Section V]{toader2004photonic}. Note that in   \cite[Section V]{toader2004photonic} the  frequency is measured in units of $2\pi$, so that the frequency-dependent dielectric constant in the coating of thickness $0.1\,r_{cp}$ in our setting is given by \
\begin{eqnarray}
\epsilon_{\mbox{\tiny coating}} (\omega) = \epsilon_1 + \Lambda^2 \frac{\omega_0^2-(\omega/2\pi)^2}{[\omega_0^2-(\omega/2 \pi)^2]^2+(\omega/2\pi)^2 \vecgamma_0^2},\label{freqdef}
\end{eqnarray}
where $\epsilon_1=7$, $\omega_0=0.489$, $\vecgamma_0=0.3$ and $\Lambda=\sqrt{1.9}$ are numerical parameters;
the dielectric constant in the spherical core of radius $0.9\,r_{cp}$ is given by $\epsilon_{\mbox{\tiny core}} = 1.592^2$; {outside the spheres} the dielectric constant is $1$. We use quadratic edge elements of the second kind to discretize the magnetic field, and cubic vertex elements to discretize the Lagrange multiplier. The mesh size requested from the mesh generator is 1/2.  Results are shown Fig.~\ref{fig2}. In that figure the vertical axis is $\omega/2\pi$, and the horizontal axis is $\alpha$ which defines the wave vector $\veck$ {by (\ref{alphadef}).}
Clearly there are differences between the {two frequency independent coefficient results in Fig.~\ref{fig2}, right panel, and  the frequency dependent coefficient results in Fig.~\ref{fig3},} as is to be expected.

\begin{figure}
\begin{center}
\begin{tabular}{cc}
\resizebox{0.4\textwidth}{!}{\includegraphics{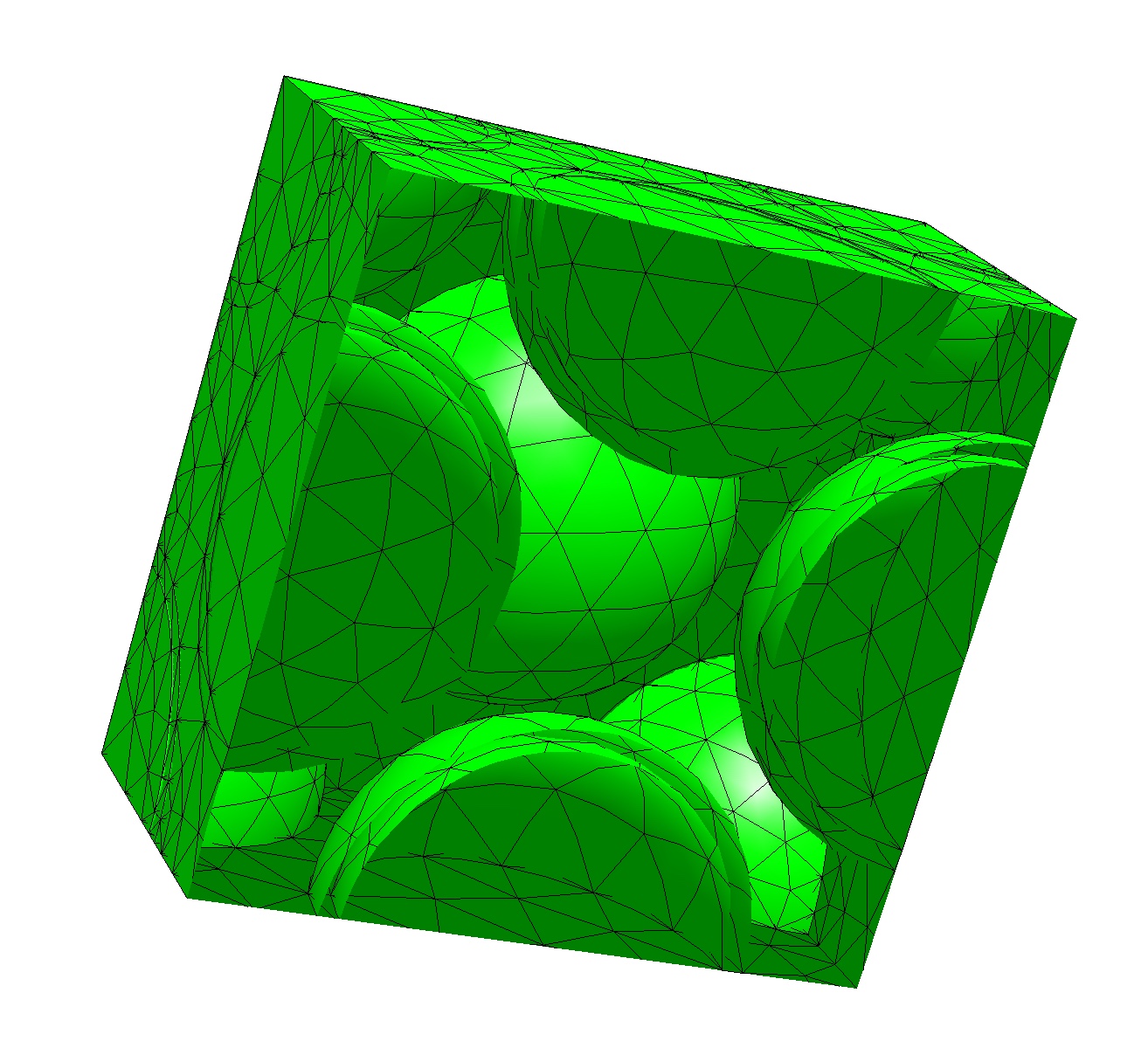}} & 
\resizebox{0.5\textwidth}{!}{\includegraphics{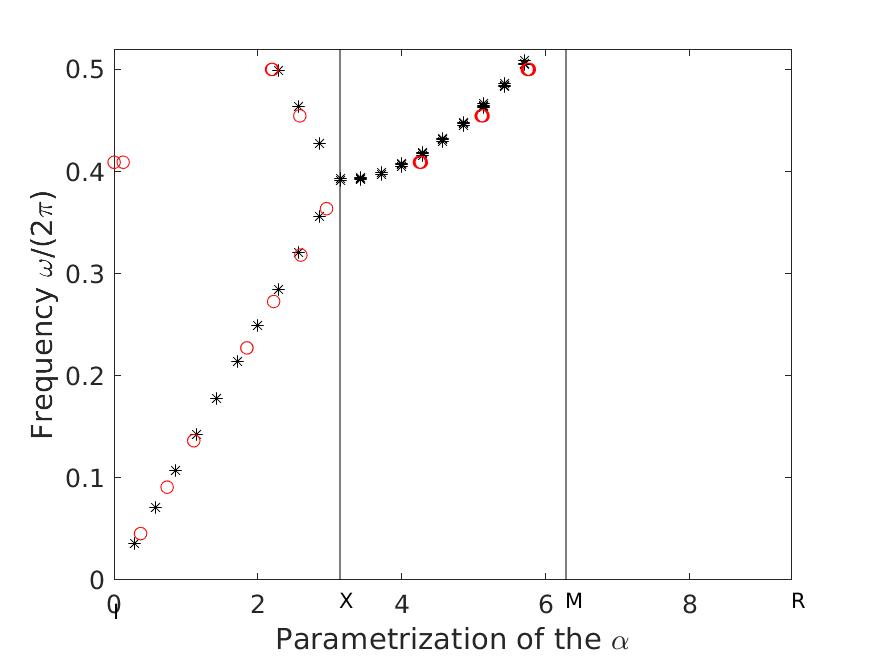}}
\end{tabular}
\end{center}
\caption{Left: A cross section of the mesh used for Example 2. Right panel: Eigenvalues for Example 2 using 
frequency independent parameters.  Points computed solving a standard finite element method are shown as ($*$) and using the linearized quadratic eigenvalue approach as (\textcolor{red}{$\circ$}). }
\label{fig2}
\end{figure}

\begin{figure}
\begin{center}
\resizebox{0.5\textwidth}{!}{\includegraphics{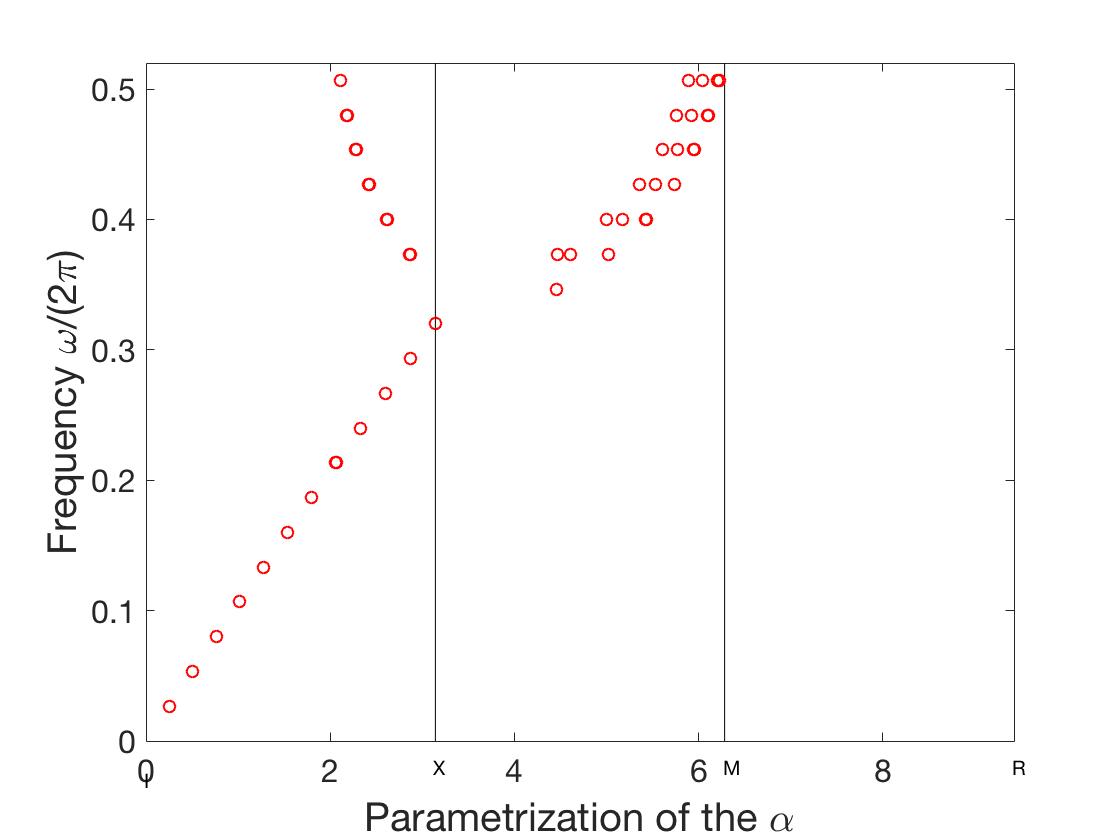}} 
\end{center}
\caption{Results for the frequency dependent coefficient defined in {(\ref{freqdef}) using the face centered lattice shown in the left panel of Fig.~\ref{fig2} and computed  using linearized quadratic eigenvalue approach} (\textcolor{red}{$\circ$}).  There are clear differences between the eigenvalues
in this figure and Fig.~\ref{fig2} right panel.}
\label{fig3}
\end{figure}
\subsection{Convergence Rate}
In our final study, we attempt to determine the convergence rate of our method.  We return to the frequency independent case in Example~1 and choose a specific point on the Bloch variety computed using the {standard approach (fixing $\veck$ and computing $\omega$)} with cubic edge elements and a fine mesh (mesh parameter $1/3$). {We} use this as the ``exact'' solution. In
particular we choose
\[
\veck=(\pi/2,0,0) \mbox{ and find } \omega=2\pi\,0.14492297.
\]
We then fix $\omega$ at the above value and solve the problem for $\veck$ using our linearized quadratic approach
with linear or quadratic edge elements.  By adjusting the mesh size requested from Netgen we can obtain different numbers of degrees of freedom and hence study the convergence as $N$, the number of degrees of freedom, increases.
Note that the meshes are not nested and therefore simply increasing $N$ may not result in a better solution (some points
in the graph are outliers).  Nevertheless the trend in Fig.~\ref{fig3} is clear.  For linear edge elements we are seeing first order convergence, while for quadratic edge elements we see quadratic convergence.  This is consistent with the expected convergence rate for a non-self adjoint eigenvalue problem using these {finite} elements.

\begin{figure}
\begin{center}
\resizebox{0.6\textwidth}{!}{\includegraphics{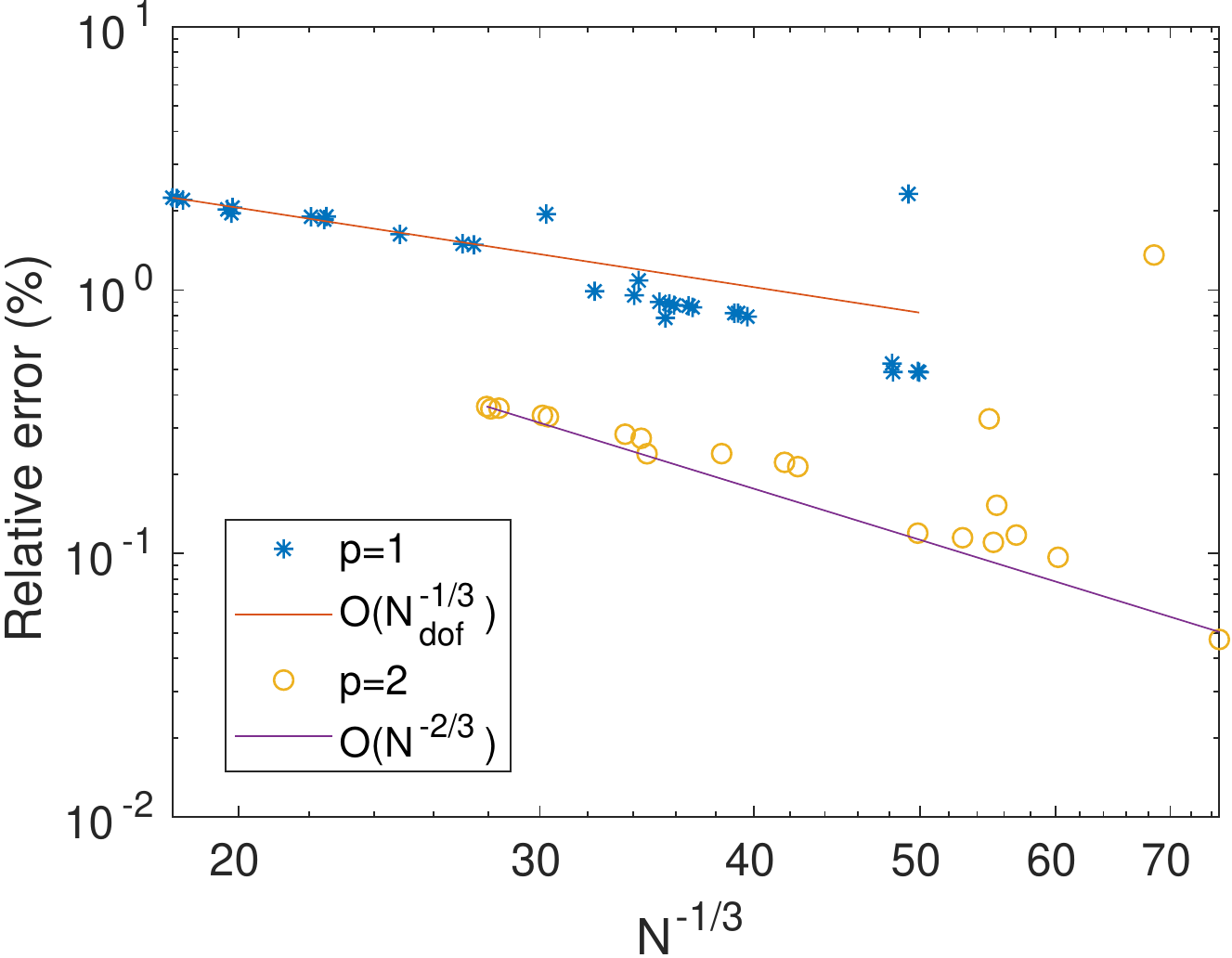}}
\end{center}
\caption{Relative error in the predicted wave vector {$\veck$} when $\omega=2\pi\,0.14492297$ against the reciprocal of the cube root of the number degrees of freedom in the problem. We show results for edge elements of  polynomial degree $p=1$ and $p=2$. Reference lines are $O(N^{-1/3})$ (linear convergence) and $O(N^{-2/3})$ (quadratic convergence).}
\label{errfig}
\end{figure}

\section{Conclusion}\label{concl}
We have shown that the problem of computing the Bloch variety for photonic crystals  having frequency dependent material
coefficients can be written as a quadratic eigenvalue problem in  a stable way.  The resulting problem can be linearized
and the Bloch variety can be found by computing the wave-vectors as a function of the angular frequency.

Much remains to be done: in particular convergence of the method has not been proved (although it is observed experimentally).   

\section*{Acknowledgements}
The research of  P.B. Monk was partially supported by the US National Science Foundation (NSF) under grant number DMS-1619904 and by the Air Force Office of Scientific Research (AFOSR) under award number FA9550-17-1-0147. 
 S. Meng was partially supported by the Air Force Office of Scientific Research under award FA9550-18-1-0131.
\bibliographystyle{plain}
\bibliography{ArxivBloch}

\begin{thebibliography}{10}

\bibitem{alagappan2013optical}
G.~Alagappan and A.~Deinega.
\newblock Optical modes of a dispersive periodic nanostructure.
\newblock {\em Progress In Electromagnetics Research}, 52:1--18, 2013.

\bibitem{SLEPc}
V.~Hernandez andJ. E.~Roman and V.~Vida.
\newblock {SLEP}c: {A} scalable and flexible toolkit for the solution of
  eigenvalue problems.
\newblock {\em ACM Transactions on Mathematical Software}, 31:351--362, 2005.

\bibitem{boffi2010finite}
D.~Boffi.
\newblock Finite element approximation of eigenvalue problems.
\newblock {\em Acta Numerica}, 19:1--120, 2010.

\bibitem{boffi2006interpolation}
D.~Boffi and L.~Gastaldi.
\newblock Interpolation estimates for edge finite elements and application to
  band gap computation.
\newblock {\em Applied numerical mathematics}, 56(10-11):1283--1292, 2006.

\bibitem{brule2016calculation}
Y.~Br{\^u}l{\'e} and B.~Gralakand~G. Dem{\'e}sy.
\newblock Calculation and analysis of the complex band structure of dispersive
  and dissipative two-dimensional photonic crystals.
\newblock {\em JOSA B}, 33(4):691--702, 2016.

\bibitem{roman18}
G.~Dem\'esy, A.~Nicolet, B.~Gralak, C.~Geuzaine, C.~Campos, and J.E. Roman.
\newblock Eigenmode computations of frequency-dispersive photonic open
  structures: {A} non-linear eigenvalue problem.
\newblock https://arxiv.org/abs/1802.02363, 2018.

\bibitem{diaz2015thz}
F.~D{\'\i}az-Monge, A.~Paredes-Ju{\'a}rez, D.A. Iakushev, N.M. Makarov, and
  F.~P{\'e}rez-Rodr{\'\i}guez.
\newblock Thz photonic bands of periodic stacks composed of resonant dielectric
  and nonlocal metal.
\newblock {\em Optical Materials Express}, 5(2):361--372, 2015.

\bibitem{dobson2000efficient}
D.C. Dobson, J.~Gopalakrishnan, and J.E. Pasciak.
\newblock An efficient method for band structure calculations in {3D} photonic
  crystals.
\newblock {\em Journal of Computational Physics}, 161(2):668--679, 2000.

\bibitem{dobson2001analysis}
D.C. Dobson and J.E. Pasciak.
\newblock Analysis of an algorithm for computing electromagnetic {B}loch modes
  using {N}ed\'el\'ec spaces.
\newblock {\em Comput. Methods Appl. Math.}, 1(2):138--153, 2001.

\bibitem{effenberger2012linearization}
C.~Effenberger, D.~Kressner, and C.~Engstr{\"o}m.
\newblock Linearization techniques for band structure calculations in absorbing
  photonic crystals.
\newblock {\em International Journal for Numerical Methods in Engineering},
  89(2):180--191, 2012.

\bibitem{engstrom2014spectral}
C.~Engstr{\"o}m.
\newblock Spectral approximation of quadratic operator polynomials arising in
  photonic band structure calculations.
\newblock {\em Numerische Mathematik}, 126(3):413--440, 2014.

\bibitem{engstrom2009spectrum}
C.~Engstr{\"o}m and M.~Richter.
\newblock On the spectrum of an operator pencil with applications to wave
  propagation in periodic and frequency dependent materials.
\newblock {\em SIAM Journal on Applied Mathematics}, 70(1):231--247, 2009.

\bibitem{engstrom2010complex}
C.~Engstr{\"o}m and M.~Wang.
\newblock Complex dispersion relation calculations with the symmetric interior
  penalty method.
\newblock {\em International journal for numerical methods in engineering},
  84(7):849--863, 2010.

\bibitem{gu2006applications}
B.-Y. Gu and L.-M. Zhaoand Y.-C. Hsue.
\newblock Applications of the expanded basis method to study the properties of
  photonic crystals with frequency-dependent dielectric functions and
  dielectric losses.
\newblock {\em Physics Letters A}, 355(2):134--141, 2006.

\bibitem{guttel2017nonlinear}
S.~G{\"u}ttel and F.~Tisseur.
\newblock The nonlinear eigenvalue problem.
\newblock {\em Acta Numerica}, 26:1--94, 2017.

\bibitem{hermann2008photonic}
D.~Hermann, M.~Diem, S.F. Mingaleev, A.~Garc{\'\i}a-Mart{\'\i}n, P.~W{\"o}lfle,
  and K.~Busch.
\newblock Photonic crystals with anomalous dispersion: Unconventional
  propagating modes in the photonic band gap.
\newblock {\em Physical Review B}, 77(3):035112, 2008.

\bibitem{ivchenko2006resonant}
E.L. Ivchenko and A.N. Poddubny.
\newblock Resonant three-dimensional photonic crystals.
\newblock {\em Physics of the Solid State}, 48(3):581--588, 2006.

\bibitem{JJWM}
J.D. Joannopoulos, S.G. Johnson, J.N. Winn, and R.D. Meade.
\newblock {\em Photonic Crystals}.
\newblock Princeton University Press, Princeton, 2nd edition, 2008.

\bibitem{MPB}
S.G. Johnson and J.D. Joannopoulos.
\newblock Block-iterative frequency-domain methods for {M}axwell's equations in
  a planewave basis.
\newblock {\em Optics Expres}, 8:173--190, 2001.
\newblock See also http://mpb.readthedocs.io/en/latest/.

\bibitem{kaso2007nonlinear}
A.~Kaso and S.~John.
\newblock Nonlinear bloch waves in metallic photonic band-gap filaments.
\newblock {\em Physical Review A}, 76(5):053838, 2007.

\bibitem{kuchment2016overview}
P.~Kuchment.
\newblock An overview of periodic elliptic operators.
\newblock {\em Bulletin of the American Mathematical Society}, 53(3):343--414,
  2016.

\bibitem{arnoldi}
R.~B. Lehoucq, D.~C. Sorensen, and C.~Yang.
\newblock {\em {ARPACK} Users Guide: {S}olution of Large-Scale Eigenvalue
  Problems with Implicitly Restarted {A}rnoldi Methods}.
\newblock SIAM, 1998.

\bibitem{ned86}
J.C. N\'{e}d\'{e}lec.
\newblock A new family of mixed finite elements in $\real^3 $.
\newblock {\em Numer. Math.}, 50:57--81, 1986.

\bibitem{park1999three}
S.~H. Park, B.~Gates, and Y.~Xia.
\newblock A three-dimensional photonic crystal operating in the visible region.
\newblock {\em Advanced Materials}, 11(6):462--466, 1999.

\bibitem{bai2013efficient}
Q.~Baiand~M. Perrin, C.~Sauvan, J.-P. Hugonin, and P.~Lalanne.
\newblock Efficient and intuitive method for the analysis of light scattering
  by a resonant nanostructure.
\newblock {\em Optics Express}, 21(22):27371--27382, 2013.

\bibitem{raman2010photonic}
A.~Raman and S.~Fan.
\newblock Photonic band structure of dispersive metamaterials formulated as a
  hermitian eigenvalue problem.
\newblock {\em Physical review letters}, 104(8):087401, 2010.

\bibitem{rybin2016inverse}
M.V. Rybin and M.F. Limonov.
\newblock Inverse dispersion method for calculation of complex photonic band
  diagram and pt symmetry.
\newblock {\em Physical Review B}, 93(16):165132, 2016.

\bibitem{schoberl1997netgen}
J.~Sch{\"o}berl.
\newblock {NETGEN} an advancing front 2{D}/3{D}-mesh generator based on
  abstract rules.
\newblock {\em Computing and Visualization in Science}, 1(1):41--52, 1997.
\newblock Available at {\tt{}https://ngsolve.org}.

\bibitem{serebryannikov2015effect}
A.E. Serebryannikov, S.~Nojima, K.B. Alici, and E.~Ozbay.
\newblock Effect of in-material losses on terahertz absorption, transmission,
  and reflection in photonic crystals made of polar dielectrics.
\newblock {\em Journal of Applied Physics}, 118(13):133101, 2015.

\bibitem{soukoulis2007negative}
C.M. Soukoulis, S.~Linden, and M.~Wegener.
\newblock Negative refractive index at optical wavelengths.
\newblock {\em Science}, 315(5808):47--49, 2007.

\bibitem{sozuer1994photonic}
H.S. S{\"o}z{\"u}er and J.P. Dowling.
\newblock Photonic band calculations for woodpile structures.
\newblock {\em Journal of Modern Optics}, 41(2):231--239, 1994.

\bibitem{SH1993}
H.S. S\"oz\"uer and J.W. Haus.
\newblock Photonic bands: {S}imple-cubic lattice.
\newblock {\em J. Opt. Soc. Am. B}, 10:296--302, 1993.

\bibitem{tisseur2001quadratic}
F.~Tisseur and K.~Meerbergen.
\newblock The quadratic eigenvalue problem.
\newblock {\em SIAM Review}, 43(2):235--286, 2001.

\bibitem{toader2004photonic}
O.~Toader and S.~John.
\newblock Photonic band gap enhancement in frequency-dependent dielectrics.
\newblock {\em Physical Review E}, 70(4):046605, 2004.

\bibitem{valentine2008three}
J.~Valentine, S.~Zhang, T.~Zentgraf, E.~Ulin-Avila, D.A. Genov, G.~Bartal, and
  Xiang X.~Zhang.
\newblock Three-dimensional optical metamaterial with a negative refractive
  index.
\newblock {\em Nature}, 455(7211):376, 2008.

\bibitem{zheng2017frequency}
F.~Zheng, J.~Tao, and A.M. Rappe.
\newblock Frequency-dependent dielectric function of semiconductors with
  application to physisorption.
\newblock {\em Physical Review B}, 95(3):035203, 2017.

\end{thebibliography}

\end{document}